\documentclass[a4paper,10pt,reqno]{smfart}
\usepackage{amssymb,amsmath,mathrsfs,graphics,graphicx,mathptm,float}
\usepackage[T1]{fontenc}
\usepackage[english, francais]{babel}

\theoremstyle{plain}
\newtheorem{thm}{Th\'eor\`eme}[section]
\newtheorem{pro}[thm]{Proposition}

\newtheorem{cor}[thm]{Corollaire}

\theoremstyle{definition}
\newtheorem*{defi}{D\'efinition}

\newtheorem{eg}[thm]{Exemple}

\newtheorem{rem}[thm]{Remarque}
\newtheorem{rems}[thm]{Remarques}

\def\og{\leavevmode\raise.3ex\hbox{$\scriptscriptstyle\langle\!\langle$~}}
\def\fg{\leavevmode\raise.3ex\hbox{~$\!\scriptscriptstyle\,\rangle\!\rangle$}}

\addtocounter{section}{0}             
\numberwithin{equation}{section}       

\begin{document}
\selectlanguage{french}

\title{Feuilletages et transformations p\'eriodiques}

\author{Dominique \textsc{Cerveau}}

\address{Membre de l'Institut Universitaire de France.
IRMAR, UMR 6625 du CNRS, Universit\'e de Rennes $1,$ $35042$ Rennes, France.  
Membre de l'ANR BLAN$06$-$3$\_$137237$}
\email{dominique.cerveau@univ-rennes1.fr}

\author{Julie \textsc{D\'eserti}}

\address{Institut de Math\'ematiques de Jussieu, Universit\'e Paris $7,$ Projet G\'eom\'etrie et Dynamique, Site Chevaleret, Case $7012,$ $75205$ Paris Cedex 13, France. 
Membre de l'ANR BLAN$06$-$3$\_$137237$}
\email{deserti@math.jussieu.fr}

\maketitle{}

\begin{altabstract}
\selectlanguage{english}
Bertini classified the birational involutions of the complex projective
plane, but his geometric approach does not allow to explicit these maps
easily. In this article, we present an effective approach to this
problem by associating to each quadratic foliation a birational
involution which is, in the generic case, a Geiser involution; this 
subject has already been covered by Geiser, Milinowski, Williams and alt.
We end by making experiences, obtaining trivolutions from some
foliations of degree $3.$

\end{altabstract}

\selectlanguage{french}

\begin{abstract}
La classification des involutions birationnelles du plan projectif complexe
remonte essentiellement \`a Bertini; cette classification de nature
g\'eom\'etrique permet difficilement la construction explicite de telles
involutions. Nous proposons une approche effective permettant
d'associer \`a tout feuilletage quadratique une involution qui, dans le cas
g\'en\'erique, est de Geiser; ce sujet a d\'ej\`a \'et\'e \'etudi\'e par
Geiser, Milinowski, Williams et alt. Nous pr\'esentons ensuite quelques exp\'eriences qui 
produisent des trivolutions \`a partir de feuilletages cubiques tr\`es sp\'eciaux.

\noindent{\it Classification math\'ematique par sujets (2010). --- 14E07, 37F75.}
\end{abstract}

\section{Introduction}

\noindent Dans ce texte on va donner une construction qui relie feuilletages
de degr\'e $2$ (resp. $3$) et involutions (resp. trivolutions) birationnelles.
 \`A un feuilletage $\mathcal{F}$ de degr\'e $2$ du plan projectif complexe, 
on peut associer une involution birationnelle: une droite g\'en\'erique de $\mathbb{P}^2(\mathbb{C})$
a deux points de tangence avec $\mathcal{F};$ l'application $\mathcal{I}_\mathcal{F}$ qui permute 
ces deux points est une involution birationnelle dite involution associ\'ee
\`a $\mathcal{F}.$ On relie les points \'eclat\'es par $\mathcal{I}_\mathcal{F}$ et les points singuliers 
de~$\mathcal{F};$ les courbes contract\'ees par $\mathcal{I}_\mathcal{F}$ et les courbes des points de
 tangence entre $\mathcal{F}$ et les pinceaux de droites passant par les points singuliers de~$\mathcal{F};$ l'adh\'erence des points d'inflexion de $\mathcal{F}$ 
et l'adh\'erence des points fixes de $\mathcal{I}_\mathcal{F}.$ On montre que l'involution associ\'ee
\`a un feuilletage quadratique g\'en\'erique de $\mathbb{P}^2(\mathbb{C})$ est une involution de 
Geiser; ceci permet d'en donner des exemples explicites. 
On s'int\'eresse en particulier au feuilletage $\mathcal{F}_J$ de Jouanolou de degr\'e~$2$ pour lequel l'involution associ\'ee~$\mathcal{I}_{\mathcal{F}_J}$ est une involution de Geiser; le groupe engendr\'e
par $\mathcal{I}_{\mathcal{F}_J}$ et le groupe d'isotropie de~$\mathcal{F}_J$ est un sous-groupe
fini d'ordre~$42,$ non lin\'earisable du groupe de Cremona. Traditionnellement les involutions de Geiser sont construites via la donn\'ee d'un pinceau de cubiques en position g\'en\'erale (\S\ref{clainv}). Comme nous l'a indiqu\'e le referee la construction explicite de telles involutions, ou les probl\`emes qui lui sont connexes, a int\'eress\'e de nombreux math\'ematiciens parmi lesquels il faut citer Weddle, Hart, Hesse, Chasles, Cayley, Geiser, Milinowski. Dans un article \'ecrit en $1920$ Williams red\'ecouvre et pr\'ecise la d\'emarche de la construction de l'involution de Geiser associ\'ee \`a sept points en position g\'en\'erale (\cite{Will}); il en donne les d\'eg\'en\'erescences, d\'eg\'en\'erescences associ\'ees \`a celles de la position des points. \'Evidemment la lecture de cet article montre l'\'etendue des connaissances sur le sujet mais pr\'esente une certaine difficult\'e due en particulier \`a l'\'evolution du langage et aux \'evidences propres \`a une \'epoque dans un domaine qui revient \`a l'ordre du jour.

\noindent R\'eciproquement \`a une involution birationnelle $\mathcal{I}=(\mathcal{I}_1,\mathcal{I}_2)$ on peut associer
le feuilletage $\mathcal{F}$ d\'ecrit en carte affine par le champ de vecteurs $\Big(x-\mathcal{I}_1(x,y)\Big)\frac{\partial}{\partial x}+\Big(y-\mathcal{I}_2(x,y)\Big)\frac{
\partial}{\partial y}.$ Ce feuilletage est de degr\'e pair $2n$ avec $n>1$ et 
chaque droite g\'en\'erique contient~$n$ orbites distinctes suivant $\mathcal{I};$ c'est par exemple
le cas pour les involutions de Bertini o\`u les orbites sont arrang\'ees en constellations 
de $4$ orbites en alignement. En degr\'e $2$ l'adh\'erence des points d'inflexion 
de $\mathcal{F}$ est contenue dans l'adh\'erence des points fixes de $\mathcal{I}_\mathcal{F};$
on verra que ce n'est plus le cas en degr\'e sup\'erieur: les courbes de points d'inflexion
peuvent \^etre permut\'ees par $\mathcal{I}_\mathcal{F}.$

\noindent Consid\'erons un feuilletage $\mathcal{F}$ de degr\'e $3$ sur $\mathbb{P}^2(\mathbb{C}).$
Toute droite g\'en\'erique de $\mathbb{P}^2(\mathbb{C})$  est tangente \`a $\mathcal{F}$ en trois
points. L'\og application\fg\, qui \'echange ces trois points est en g\'en\'eral multivalu\'ee;
on donne un crit\`ere qui assure que cette application est birationnelle ce qui permet de produire des exemples explicites. 
On consid\`ere en particulier le cas des feuilletages homog\`enes g\'en\'eriques; 
l'\'eventuelle trivolution associ\'ee est n\'ecessairement de Jonqui\`eres. On produit des exemples qui permettent de construire des $3$-tissus hexagonaux sur $\mathbb{C}^2.$

\medskip

\subsection*{Remerciements} Dans \cite{Will} on trouve le passage suivant: \og However, Professor H. S. White, to whom this paper was refered, pointed out the subject had been well covered by Geiser and Milinowski\fg . Nous remercions le referee qui nous a indiqu\'e cette litt\'erature ancienne et dont le commentaire \'etait proche de celui de White.

\bigskip

\section{Quelques d\'efinitions et notations}

\subsection{Transformations birationelles}

\noindent Une {\sl transformation rationnelle} $f\colon\mathbb{P}^2(\mathbb{C})\dashrightarrow\mathbb{P}^2(\mathbb{C})$ du plan projectif complexe dans lui-m\^eme est de la forme 
$$(x:y:z)\mapsto (f_0(x,y,z):f_1(x,y,z):f_2(x,y,z)),$$
les $f_i$ d\'esignant des polyn\^omes homog\`enes de m\^eme degr\'e sans
facteur commun. Le {\sl degr\'e} de~$f$ est, par d\'efinition, le degr\'e des $f_i.$
Une {\sl transformation birationnelle} est une transformation rationnelle qui 
admet un inverse, lui-m\^eme rationnel.
Le {\sl groupe de Cremona}, not\'e $\mathrm{Bir}(\mathbb{P}^2(\mathbb{C})),$ est 
le groupe des transformations birationnelles du plan projectif complexe.
Le \textsl{lieu d'ind\'etermination} de $f,$ ou encore l'ensemble des \textsl{points 
\'eclat\'es} par $f,$ est le lieu d'annulation des~$f_i;$ on le d\'esigne par~$\mathrm{Ind}(f).$
Le \textsl{lieu exceptionnel} de $f$ est l'ensemble des z\'eros du d\'eterminant jacobien 
de $f;$ on le note~$\mathrm{Exc}(f).$ On dit que les \'el\'ements de $\mathrm{Exc}(f)$ sont 
les \textsl{courbes contract\'ees} par $f.$ D\'esignons par $\mathrm{Fix}
(f)$ l'adh\'erence des points fixes de $f;$ 
c'est une union de courbes et de points isol\'es.

\begin{eg}
La transformation dite {\sl involution de Cremona}
\begin{align*}
& \sigma\colon\mathbb{P}^2(\mathbb{C})\dashrightarrow\mathbb{P}^2(\mathbb{C}), && (x:y:z)\mapsto(yz:xz:xy)
\end{align*}
\hspace{-0.2mm} est birationnelle de degr\'e $2.$ On v\'erifie que
\begin{align*}
&\mathrm{Ind}(\sigma)=\{(1:0:0),\,(0:1:0),\,
(0:0:1)\}, &&\mathrm{Exc}(\sigma)=\{x=0\}\cup\{y=~0\}\cup\{z=~0\},
\end{align*}
$$\mathrm{Fix}(\sigma)=\{(1:1:1),\,(1:1:-1),\,(1:-1:1),\,(-1:1:1)\}.$$
\end{eg}

\noindent Le {\sl groupe de Jonqui\`eres} $\mathrm{dJ}$ est le sous-groupe maximal de
$\mathrm{Bir}(\mathbb{P}^2(\mathbb{C}))$ form\'e des transformations pr\'eservant la
fibra\-tion~$y=$ cte, {\it i.e.} 
\begin{align*}
\mathrm{dJ}\simeq\mathrm{PGL}_2(\mathbb{C}(y))\rtimes
\mathrm{PGL}_2(\mathbb{C})=\left\{\left(\frac{a(y)x+b(y)}{c(y)x+d(y)},\frac{
\alpha y+\beta}{\gamma y+\delta}\right)\,\Big\vert\,\left[\begin{array}{cc}a & b\\  c&d\end{array}\right]\in\mathrm{PGL}_2(\mathbb{C}(y)),\,\left[\begin{array}{cc}\alpha &\beta\\ \gamma&\delta\end{array}\right]\in\mathrm{PGL}_2(\mathbb{C})\right\}.
\end{align*}
On appelle {\sl transformation de Jonqui\`eres} toute
transformation birationnelle pr\'eservant une fibration rationnelle, ceci \'etant justifi\'e
par le fait que toute fibration rationnelle est birationnellement conjugu\'ee \`a la
fibration $y=$~cte. Lorsque la matrice $\left[\begin{array}{cc}\alpha &\beta\\ \gamma&\delta\end{array}\right]$ repr\'esente l'identit\'e, on dit que la transformation correspondante respecte la fibration fibre \`a fibre.

\bigskip

\subsection{Feuilletages}

\noindent Un {\sl feuilletage holomorphe} $\mathcal{F}$ de codimension $1$ et de
degr\'e $\nu$ sur~$\mathbb{P}^2(\mathbb{C})$ est d\'efini par une $1$-forme du type
$$\omega=u(x,y,z)\mathrm{d}x+v(x,y,z)\mathrm{d}y+w(x,y,z)\mathrm{d}z,$$ o\`{u} $u,$ 
$v$ et $w$ sont des polyn\^{o}mes homog\`enes de
degr\'e $\nu + 1$ sans composante commune v\'erifiant
l'identit\'e d'Euler: $$xu(x,y,z)+yv(x,y,z)+zw(x,y,z)=0.$$ 
Le {\sl lieu singulier}
$\mathrm{Sing}(\mathcal{F})$ de $\mathcal{F}$ est le projectivis\'e du lieu
singulier de~$\omega$
$$\mathrm{Sing}(\omega)=\{(x,y,z)\in\mathbb{C}^3\,\vert \, u(x,y,z)=v(x,y,z)=w(x,y,z)=0\}.$$
En restriction \`a la carte affine $z=1$ la $1$-forme $\omega$ s'\'ecrit
$u(x,y,1)\mathrm{d}x+v(x,y,1)\mathrm{d}y=\widetilde{u}\mathrm{d}x+\widetilde{v}
\mathrm{d}y+\phi(x\mathrm{d}y-y\mathrm{d}x)$
o\`u $\widetilde{u}$ et $\widetilde{v}$ sont des polyn\^{o}mes de degr\'e au plus $\nu$ et $\phi$ un
polyn\^ome homog\`ene de degr\'e~$\nu.$

\noindent Soient $\mathcal{F}$ un feuilletage de degr\'e $\nu$ sur 
le plan projectif complexe, $\mathcal{D}$ une droite g\'en\'erale 
et $p$ un point de $\mathcal{D}$ non singulier pour $\mathcal{F}.$ 
On dit que $\mathcal{F}$ est {\sl transverse} \`a $\mathcal{D}$ en~$p$
si la feuille $\mathcal{L}_p$ de $\mathcal{F}$ en~$p$ est transverse
\`a $\mathcal{D}$ en~$p;$ sinon on dit que $p$ est un {\sl point de 
tangence} entre~$\mathcal{D}$ et $\mathcal{F}.$ Le {\sl degr\'e}~$\nu$
de~$\mathcal{F}$ est exactement le nombre de points de tangence
entre $\mathcal{D}$ et $\mathcal{F}.$

\noindent La classification des feuilletages de degr\'e $0$ ou $1$ sur
$\mathbb{P}^2(\mathbb{C})$ est connue depuis le~XIX$^{\text{\`eme}}$ si\`ecle~(\cite{J}).
Un feuilletage de degr\'e $0$ sur $\mathbb{P}^2(\mathbb{C})$ est un pinceau de droites. Tout
feuilletage de degr\'e $1$ sur le plan projectif complexe poss\`ede trois
singularit\'es compt\'ees avec multiplicit\'e, a, au moins, une droite
invariante et est donn\'e par une $1$-forme ferm\'ee rationnelle (dit 
autrement il existe un polyn\^ome homog\`ene $P$ tel que $\omega/P$
soit ferm\'ee); les
feuilles sont les composantes connexes des \og niveaux\fg\, d'une primitive
de cette $1$-forme.
Pour~$\nu\geq2$ peu de propri\'et\'es ont \'et\'e \'etablies, si ce n'est la non existence g\'en\'erique de courbe invariante~(\cite{J, CLN}).

\smallskip

\noindent Un point r\'egulier $m$ de $\mathcal{F}$ est dit {\sl d'inflexion} (pour
$\mathcal{F}$) si $\mathcal{L}_m$ a un point 
d'inflexion en $m;$ on d\'esigne par~$\mathrm{Flex}(\mathcal{F})$
l'adh\'erence de ces points. Pr\'esentons un moyen de d\'eterminer cet 
ensemble $\mathrm{Flex}(\mathcal{F})$ donn\'e dans \cite{Pe}.
Soit $\mathrm{Z}=E\frac{\partial}{\partial x}+F\frac{\partial}
{\partial y}+G\frac{\partial}{\partial z}$ un champ de vecteurs homog\`ene
sur $\mathbb{C}^3$ non colin\'eaire au champ radial $\mathrm{R}=x
\frac{\partial}{\partial x}+y\frac{\partial}{\partial y}+z\frac{\partial}{\partial z}$
d\'ecrivant~$\mathcal{F},$ {\it i.e.} tel que 
$$(\Diamond)\hspace{1cm} \omega=i_\mathrm{R}i_\mathrm{Z}\mathrm{d}x\wedge 
\mathrm{d}y\wedge \mathrm{d}z.$$ D\'efinissons le polyn\^ome $\mathcal{H}$ par
$$\mathcal{H}(x,y,z)=\det\left[\begin{array}{ccc}
x & E & \mathrm{Z}(E)\\
y & F & \mathrm{Z}(F)\\
z & G & \mathrm{Z}(G)
\end{array}
\right];$$ notons que $\mathcal{H}$ ne d\'epend pas des choix de $\omega$ et du champ
de vecteurs $\mathrm{Z}$ satisfaisant ($\Diamond$), tout du moins \`a constante multiplicative pr\`es.
D'apr\`es~\cite{Pe} le lieu des z\'eros de $\mathcal{H}$
est constitu\'e de $\mathrm{Flex}(\mathcal{F})$ et de 
l'ensemble des droites invariantes par $\mathcal{F}.$ 

\smallskip

\noindent Soit $\mathcal{F}$ un feuilletage de degr\'e $\nu$ sur $\mathbb{P}^2(\mathbb{C}).$
L'application de Gauss est l'application rationnelle $\mathcal{G}$ de $\mathbb{P}^2(\mathbb{C})$ 
dans~$\check{\mathbb{P}}^2(\mathbb{C})$ qui \`a 
un point r\'egulier $m$ associe la tangente $\mathrm{T}_m\mathcal{L}_m$ \`a 
$\mathcal{L}_m$ en $m.$ Le lieu exceptionnel de $\mathcal{G}$
co\"incide avec les droites invariantes par $\mathcal{F}$ et les points
d'ind\'etermination de $\mathcal{G}$ sont les points singuliers de
$\mathcal{F}.$ Un point $m$ est dit {\sl g\'en\'erique pour 
$\mathcal{F}$} s'il est r\'egulier et si~$\mathrm{T}_m\mathcal{L}_m$
a exactement $\nu$ contacts avec $\mathcal{F}.$ L'ensemble des points g\'en\'eriques
pour $\mathcal{F}$ est le compl\'ement de~$(\mathcal{H}=0).$ Notons
$\Lambda$ le lieu des z\'eros de $\mathcal{H}$ et $\Lambda'$ son 
image par $\mathcal{G}.$ Soit $\mathcal{D}$ un point de 
$\check{\mathbb{P}}^2(\mathbb{C})\setminus\Lambda';$ la fibre~$\mathcal{G}^{-1}(
\{\mathcal{D}\})$ contient exactement $\nu$ points et la 
restriction 
$$\mathcal{G}_{\vert\mathbb{P}^2(\mathbb{C})\setminus\Lambda}\colon\mathbb{P}^2(\mathbb{C})\setminus\Lambda\to\check{\mathbb{P}}^2(\mathbb{C})\setminus\Lambda'$$
de $\mathcal{G}$ \`a $\mathbb{P}^2(\mathbb{C})\setminus
\Lambda$ est un rev\^etement. On peut se demander \`a
quelles conditions un tel rev\^etement poss\`ede des 
automorphismes qui sont birationnels; on donne dans la suite 
quelques \'el\'ements de r\'eponse \`a cette question.

\bigskip

\section{Feuilletages quadratiques et involutions birationnelles}

\subsection{Construction de l'involution et premi\`eres propri\'et\'es}\label{constr}

\noindent \`A tout feuilletage $\mathcal{F}$ de degr\'e $2$ 
d\'efini sur le plan projectif complexe on peut associer une 
involution birationnelle $\mathcal{I}_\mathcal{F}.$ En effet consid\'erons un point 
g\'en\'erique $m$ pour $\mathcal{F};$ le feuilletage \'etant 
quadratique~$\mathrm{T}_m\mathcal{L}_m$ est tangente \`a $\mathcal{F}$
en un second point $p,$ l'involution $\mathcal{I}_\mathcal{F}$ est la transformation 
qui \'echange ces deux points. Plus pr\'ecis\'ement supposons que $\mathcal{F}$ 
soit d\'ecrit par le champ de vecteurs $\mathrm{X}.$ 
L'image par $\mathcal{I}_\mathcal{F}$ d'un point g\'en\'erique pour $\mathcal{F}$ est le point
$m+s\mathrm{X}(m)$ o\`u $s$ est l'unique param\`etre non nul pour 
lequel $\mathrm{X}(m)$ et~$\mathrm{X}(m+s\mathrm{X}(m))$ sont colin\'eaires.

\noindent Soient $q$ un point singulier de $\mathcal{F}$ et $\mathcal{P}(q)$
le pinceau de droites passant par $q.$ La courbe des points de 
tangen\-ce~$\mathrm{Tang}(\mathcal{F},\mathcal{P}(q))$
entre les feuilletages $\mathcal{F}$ et $\mathcal{P}(q)$ est 
contract\'ee par $\mathcal{I}_\mathcal{F}$ sur~$q.$  On constate que toutes les 
courbes contract\'ees sont de ce type. 
Par ailleurs notons $\mathrm{Inv}(\mathcal{F})$ l'ensemble
des droites invariantes par $\mathcal{F}.$ On a l'inclusion 
$\mathrm{Flex}(\mathcal{F})\subset\mathrm{Fix}(\mathcal{I}_\mathcal{F}).$
On verra plus loin des exemples o\`u certaines droites de 
$\mathrm{Inv}(\mathcal{F})$ sont dans $\mathrm{Fix}(\mathcal{I}_\mathcal{F})$
et d'autres sont contract\'ees par~$\mathcal{I}_\mathcal{F}.$ 

\noindent Pour $\mathcal{F}$ g\'en\'erique $\mathcal{I}_\mathcal{F}$ a sept points
d'ind\'etermination correspondant aux sept points singuliers
du feuilletage et sept courbes contract\'ees qui sont 
des cubiques \`a point double; on reviendra plus loin sur ce point.

\begin{defi}
Soit $\mathcal{F}$ un feuilletage sur $\mathbb{P}^2(\mathbb{C}).$
Le sous-groupe de $\mathrm{Aut}(\mathbb{P}^2(\mathbb{C}))$
qui pr\'eserve~$\mathcal{F}$ s'appelle le {\sl groupe d'isotropie}
de $\mathcal{F};$ c'est un sous-groupe alg\'ebrique de $\mathrm{Aut}(\mathbb{P}^2(\mathbb{C})).$ On le note 
$\mathrm{Iso}(\mathcal{F})$
\begin{align*}
\mathrm{Iso}(\mathcal{F})=\{ \varphi\in\mathrm{Aut}(\mathbb{P}^2(\mathbb{C}))\,
\vert\, \varphi^*\mathcal{F}=\mathcal{F}\}.
\end{align*} 
\end{defi}
\begin{rem}
Soient $\mathcal{F}$ un feuilletage quadratique et $\mathcal{I}_\mathcal{F}$
l'involution associ\'ee; $\mathcal{I}_\mathcal{F}$ commute \`a tous les 
\'el\'ements de~$\mathrm{Iso}(\mathcal{F}).$
\end{rem}

\bigskip

\subsection{Classification des involutions
birationnelles du plan projectif complexe}\label{clainv} 

\noindent Avant d'\'enoncer le r\'esultat de Bertini repris par Bayle
et Beauville, introduisons trois types
d'involutions qui, comme on le verra, jouent un r\^ole tr\`es particulier.

\noindent Soient 
$p_1,$ $\ldots,$ $p_7$ sept points de $\mathbb{P}^2(\mathbb{C})$ en position 
g\'en\'erale, c'est-\`a-dire satisfaisant les conditions suivantes: il n'y en 
a pas trois align\'es et il n'y en a pas six sur une conique. D\'esignons par~$L$ le syst\`eme lin\'eaire de cubiques passant
par les~$p_i;$ il est de dimension $2.$ Soit $p$ un point g\'en\'erique
de $\mathbb{P}^2(\mathbb{C});$ consid\'erons le pinceau $L_p$
constitu\'e des \'el\'ements de~$L$ passant
par $p.$ Un pinceau de cubiques g\'en\'eriques ayant neuf points 
base, on d\'efinit par $\mathcal{I}_G(p)$ le neuvi\`eme point base de~$L_p.$ 
L'involution $\mathcal{I}_G=\mathcal{I}_G(p_1,\ldots,p_7)$ qui \`a $p$ associe $\mathcal{I}_G(p)$ ainsi construite est 
appel\'ee \textsl{involution de Geiser}. On peut v\'erifier qu'une telle 
involution est birationnelle de degr\'e $8$ et que ses  
points fixes forment une courbe non hyperelliptique de genre~$3,$
de degr\'e $6$ avec $7$ points doubles ordinaires qui se trouvent en les $p_i.$
Le lieu exceptionnel d'une involution de Geiser est constitu\'e 
de sept cubiques passant par les sept points d'ind\'etermination de 
$\mathcal{I}_G$ et singuli\`ere en l'un des sept points (cubiques \`a point double). 

\noindent Consid\'erons l'ensemble $\mathcal{S}$ des sextiques
ayant huit points doubles $p_1,$ $\ldots,$ $p_8$ en position 
g\'en\'erale. Fixons un 
point $m$ de~$\mathbb{P}^2(\mathbb{C});$ le pinceau des \'el\'ements 
de $\mathcal{S}$ ayant un point double en $m$ contient un 
dixi\`eme point double $m'.$ L'involution qui \'echange les 
points $m$ et $m'$ est une \textsl{involution de Bertini} que l'on 
note $\mathcal{I}_B=\mathcal{I}_B(p_1,\ldots,p_8).$ Les points fixes de~$\mathcal{I}_B$
forment une courbe non hyperelliptique de genre~$4,$ de degr\'e~$9$ 
avec des points triples en les $p_i,$ dont la normalis\'ee est isomorphe \`a une
intersection non singuli\`ere d'une surface cubique et d'un c\^one quadratique
dans $\mathbb{P}^3(\mathbb{C}).$ 

\noindent Pour finir introduisons les involutions de Jonqui\`eres. Soit 
$\mathcal{C} $ une courbe irr\'eductible de de\-gr\'e~$\nu\geq~3.$ On suppose que~$\mathcal{C}$ poss\`ede un unique point singulier $p$ qui soit de plus un point multiple ordinaire de multiplicit\'e~$\nu-~2.$ Au couple~$(\mathcal{C},p)$ on va associer une involution 
birationnel\-le~$\mathcal{I}_{\mathrm{dJ}}$ qui fixe la courbe $\mathcal{C}$ et pr\'eserve les
droites passant par~$p.$ Soit~$m$ un point g\'en\'erique de~$\mathbb{P}^2(\mathbb{C});$ notons~$q_m,$ $r_m$ les deux points d'intersection, distincts de $p,$ de la 
droite~$(pm)$ avec $\mathcal{C}.$ La transformation $\mathcal{I}_{\mathrm{dJ}}$ associe au 
point $m$ le conjugu\'e harmonique de $m$ sur la droite $(pm)$
par rapport \`a~$q_m$ et~$r_m;$ dit autrement le point $\mathcal{I}_{\mathrm{dJ}}(m)$ v\'erifie
la propri\'et\'e suivante: le birapport de $m,$ $\mathcal{I}_{\mathrm{dJ}}(m),$ $q_m$ et~$r_m$ vaut~$-1.$ 
La transformation~$\mathcal{I}_{\mathrm{dJ}}$ est une \textsl{involution de Jonqui\`eres}
de degr\'e $\nu$ centr\'ee en~$p$ et pr\'eservant~$\mathcal{C}.$ L'ensemble des points fixes de $\mathcal{I}_{\mathrm{dJ}}$
est pr\'ecis\'ement la courbe $\mathcal{C}$ qui est de genre $\nu-2$ pour $\nu\geq 3.$ 
Pour $\nu=2$ la courbe $\mathcal{C}$ est une conique lisse; on 
fait alors la m\^eme construction en choisissant un point $p$ ext\'erieur \`a~$\mathcal{C}.$

\bigskip

\begin{defi}
On dira qu'une involution est de {\sl type projectif} si elle est birationnellement
conjugu\'ee \`a une involution projective, de {\sl type Jonqui\`eres} si elle est 
birationnellement conjugu\'ee \`a une involution de Jonqui\`eres; de m\^eme 
on parlera d'involution de {\sl type Bertini}, resp. de {\sl type Geiser}.
\end{defi}

\noindent L'\'enonc\'e suivant donne la classification des involutions birationnelles.

\begin{thm}[\cite{Ber, BaBe}]
{\sl Une involution birationnelle est de l'un des quatre types: projectif,
Jonqui\`eres, Bertini ou Geiser.}
\end{thm}

\noindent Dans \cite{BaBe} les auteurs montrent aussi que les classes 
de conjugaison des involutions de $\mathrm{Bir}(\mathbb{P}^2(\mathbb{C}))$ sont 
uniquement d\'etermin\'ees par le type birationnel des courbes de leurs points 
fixes.

\bigskip

\subsection{Feuilletages quadratiques g\'en\'eriques de $\mathbb{P}^2(\mathbb{C})$}

\noindent Rappelons qu'un feuilletage quadratique a sept points singuliers
compt\'es avec multiplicit\'e; de plus si on se donne sept points en position
g\'en\'erale (dans le m\^eme sens que 
pr\'ec\'edemment) il existe un et un seul feuilletage quadratique ayant ces sept
points pour points singuliers~(\cite{GMK}).

\begin{thm}\label{geiser}
{\sl Soient $p_1,$ $\ldots,$ $p_7$ sept points de
$\mathbb{P}^2(\mathbb{C})$ en position g\'en\'erale. Soient $\mathcal{F}$ le feuilletage
quadratique de~$\mathbb{P}^2(\mathbb{C})$ dont le lieu singulier
est constitu\'e des $p_i$ et $\mathcal{I}_G$
l'involution de Geiser associ\'ee aux $p_i.$ L'involution~$\mathcal{I}_\mathcal{F}$ associ\'ee \`a $\mathcal{F}$ et $\mathcal{I}_G$ co\"incident.}
\end{thm}

\begin{cor}
{\sl L'involution associ\'ee \`a un feuilletage quadratique g\'en\'erique sur $\mathbb{P}^2(\mathbb{C})$ est une involution de Geiser.}
\end{cor}

\begin{proof}[{\sl D\'emonstration du th\'eor\`eme \ref{geiser}}]
Soit $\Omega$ une $1$-forme d\'efinissant $\mathcal{F}$
$$\Omega=u\mathrm{d}x+v\mathrm{d}y+w\mathrm{d}z$$
avec $u,$ $v,$ $w$ des polyn\^omes homog\`enes de degr\'es
$3$ en $x,$ $y$ et $z$ satisfaisant l'identit\'e d'Euler, {\it i.e.} 
$xu+yv+zw=0.$

\noindent L'ensemble des cubiques passant par les sept 
points singuliers est d'apr\`es l'hypoth\`ese de g\'en\'ericit\'e
un plan projectif. Il s'identifie au projectivis\'e $\mathbb{P}(\mathrm{E})$
de l'espace vectoriel $\mathrm{E}=\{\lambda_1u+\lambda_2v+
\lambda_3w\, \vert\,\lambda_i\in~\mathbb{C}\}.$

\noindent Reprenons l'involution de Geiser $\mathcal{I}_G$ associ\'ee au $7$-uplet
de points $P=(p_1,\ldots,p_7).$ Si $m$
est un point g\'en\'erique du plan projectif complexe, l'ensemble
des \'el\'ements de $\mathbb{P}(\mathrm{E})$ passant par $m$ est un 
pinceau $\mathbb{P}(\mathrm{E}(m))$ avec 
$$\mathrm{E}(m)=\Big\{\lambda_1u+\lambda_2v+\lambda_3w\,\vert\,
\lambda_i\in\mathbb{C},\,\lambda_1u(
\widetilde{m})+\lambda_2v(\widetilde{m})+\lambda_3w(
\widetilde{m})=0\Big\}$$
o\`u $\widetilde{m}$ est un relev\'e de $m$ \`a $\mathbb{C}^3
\setminus\{0\}.$ Ce pinceau contient les $p_i,$ $m$ et
donc un autre point base not\'e $m'.$ L'involution de 
Geiser  est la transformation qui \`a $m$ associe
$m'.$ Remarquons que~$\mathrm{E}(m)=~\mathrm{E}(m')$ et que par suite
$\Omega(\widetilde{m})$ et $\Omega(\widetilde{m}')$ 
sont~$\mathbb{C}$-colin\'eaires, {\it i.e.} les plans $\ker\Omega(
\widetilde{m})$ et~$\ker\Omega(\widetilde{m}')$ sont identiques.
Notons aussi que l'identit\'e d'Euler assure que~$\widetilde{m}$ et $\widetilde{m}'$ appartiennent \`a ce plan.
Il en r\'esulte que $\ker\Omega(\widetilde{m})=\ker\Omega(\widetilde{m}')$
est le plan d\'efini par $\widetilde{m}$ et $\widetilde{m}'.$
Ceci se traduit sur~$\mathbb{P}^2(\mathbb{C})$ de la 
fa\c{c}on suivante: la droite $(mm')$ dans $\mathbb{P}^2(\mathbb{C})$ est 
tangente \`a $\mathcal{F}$ en~$m$
et $m';$ par cons\'equent l'involution de Geiser~$\mathcal{I}_G$
est exactement l'involution $\mathcal{I}_\mathcal{F}$ associ\'ee au
feuilletage~$\mathcal{F}.$
\end{proof}

\noindent Les deux remarques qui suivent pr\'ecisent ce qui est annonc\'e au 
\S\ref{constr}.

\begin{rem}
Si $\mathcal{F}$ est un feuilletage quadratique,
le polyn\^ome $\mathcal{H}$ associ\'e \`a $\mathcal{F}$ est 
de degr\'e $6;$ g\'en\'eriquement un tel $\mathcal{F}$
n'a pas de droite invariante (\cite{J, CLN}) donc 
$\mathrm{Flex}(\mathcal{F})$ et $(\mathcal{H}=0)$ co\"incident. D'autre part comme
$\mathcal{I}_\mathcal{F}$ est de type Geiser sa courbe de points fixes est une 
courbe irr\'eductible de degr\'e $6,$ ce qui force l'inclusion $\mathrm{Flex}
(\mathcal{F})\subset~\mathrm{Fix}(\mathcal{I}_\mathcal{F})$ \`a \^etre une \'egalit\'e.
\end{rem}

\begin{rem}
Supposons que l'origine de la carte~$z=1$ soit un point
singulier de $\mathcal{F}.$ Comme $\mathcal{F}$ est g\'en\'erique, il est, pour 
un bon choix de coordonn\'ees affines, d\'ecrit par 
\begin{align*}
&\mathrm{X}=(x+\ldots)\frac{\partial}{\partial x}+(ay+\ldots)\frac{\partial}{\partial y}, &&
a\not\in\{0,\, 1\}.
\end{align*}
\hspace{-0.2mm}Soit $t\mapsto(x(t),y(t))$ une courbe int\'egrale locale de $\mathrm{X}$ 
qui n'est pas une droite; les points d'inflexion sur cette courbe int\'egrale
sont donn\'es par $\dot{x}\ddot{y}-\dot{y}\ddot{x}=0.$ L'ensemble
$\mathrm{Flex}(\mathcal{F})$ est donc d\'ecrit au voisinage de l'origine
par $a(a-1)xy+\ldots=0.$ Il en r\'esulte que g\'en\'eriquement  $\mathrm{Flex}(\mathcal{F}),$
et donc $\mathrm{Fix}(\mathcal{I}_\mathcal{F}),$
est une sextique avec $7$ points nodaux ce qui est un r\'esultat classique.
\end{rem}

\bigskip

\subsection{Exemples explicites d'involution de Geiser}

\noindent Soit $\mathcal{F}$ un feuilletage quadratique g\'en\'erique sur $\mathbb{P}^2(\mathbb{C});$ 
on peut supposer \`a conjugaison lin\'eaire pr\`es que $(1:0:0),$ $(0:1:0),$ $(0:0:1)$ 
et~$(1:1:1)$ sont singuliers pour $\mathcal{F}.$
Le feuilletage $\mathcal{F}$ est alors d\'efini par
$$(\varepsilon x^2y+a x^2+bxy+cx+ey)\frac{\partial}{\partial x}+(\varepsilon xy^2+Ay^2+Bxy+Cx+Ey)\frac{\partial}{\partial y}.$$
Puisque g\'en\'eriquement la droite \`a l'infini n'est pas pr\'eserv\'ee par $\mathcal{F}$
le coefficient $\varepsilon$ est non nul et on se ram\`ene \`a $\varepsilon=1;$ de plus on a
$1+a+b+c+e=1+A+B+C+E=0,$ ces deux derni\`eres conditions 
assurant que $(1:1:1)$ est singulier.

\noindent L'involution associ\'ee $\mathcal{I}_\mathcal{F}$ s'\'ecrit $\left(\frac{U_1}{TU_2},\frac{V_1}{TV_2}\right)$
avec
\begin{small}
\begin{align*}
&U_1=(B-a)\Big(E-A(B-a)\Big)x^3y^2
+\Big(E(aB-a^2+C)+2AC(a-B)\Big)x^3y
+C(aE-AC)x^3+e(AE-bE-cA+eB)y^3\\
&\hspace{0.8cm}+\Big(2(A^2c-eE-bcA-A^2E)-aeA+ce+beB+3bAE-b^2E\Big)xy^3
+(cE-eC)(E-c)xy+e(eC-cE)y^2\\
&\hspace{0.8cm}+\Big(aeB-eC-eB^2+E^2-cE+bBE+2(bAC-acA-A^2C-abE-ABE+cAB)+3aAE\Big)x^2y^2
+C(cE-eC)x^2\\
&\hspace{0.8cm}+\Big(2(cAC-eCB-acE-ACE)+aE^2+cBE+bCE+aeC\Big)x^2y
+\Big(2A(A-b)(a-B)+e(a-B)+(A-b)E\Big)x^2y^3
\\
&\hspace{0.8cm}+(b-A)\Big(e-A(b-A)\Big)xy^4
+(beC-aeE-c^2A+bE^2-AE^2+3cAE+ceB-eBE-2bcE)xy^2+e(A^2-bA+e)y^4
\end{align*}
\begin{align*}
&V_1=
\Big(2a(a-B)(b-A)-ac-AC+bC+cB\Big)x^3y^2+(A-b)\Big(a(A-b)+c\Big)x^2y^3
+(a-B)\Big(C+a(a-B)\Big)x^4y\\
&\hspace{0.8cm}+\Big(2(aBE+cC-a^2E+a^2c)-3acB-CE+aAC-bBC+cB^2\Big)x^3y+C(cE-eC)x^2+(E-c)(cE-eC)xy\\
&\hspace{0.8cm}+\Big(2(aAE-aeB+abc-abE+cAB+a^2e)+b^2C+eC-bcB-c^2-3acA-bAC+cE\Big)x^2y^2-C(C-aB+a^2)x^4\\
&\hspace{0.8cm}+(2cBE+aE^2-bCE-3acE-c^2B+bcC+cAC-eCB+ac^2)x^2y+C(aE-bC-ac+cB)x^3+e(ae-cA)y^3\\
&\hspace{0.8cm}+\Big(2(ace+beC-aeE+cAE)-c^2A-ceB-eAC-bcE\Big)xy^2+
\Big(cA^2-ce-bcA+2ae(b-A)\Big)xy^3+e(eC-cE)y^2,
\end{align*}
\begin{align*}
&U_2=xy^2+Ay^2+Bxy+Cx+Ey, && V_2=x^2y+ax^2+bxy+cx+ey,
\end{align*}
\end{small}
\begin{footnotesize}
\begin{align*}
&T=(C-aB+a^2)x^2
+(E-c-AB+ab)xy
+(bC-cB-AC+ac)x+(bA-A^2-e)y^2
+(bE-eB-AE+ae)y,
\end{align*}
\end{footnotesize}

\noindent Pour des choix g\'en\'eriques des coefficients ceci permet d'avoir des exemples explicites d'involutions de Geiser \og normalis\'ees\fg\, au sens o\`u $4$ des points d'ind\'etermination parmi les $7$ sont fix\'es comme pr\'ec\'edemment.
Voici comment on proc\`ede. On choisit dans $\mathbb{C}^2\subset\mathbb{P}^2(\mathbb{C})$ trois nouveaux points $m_1,$ $m_2,$ $m_3$ de sorte que dans nos sept points il n'y ait pas de configurations de $3$ points en alignement. On cherche alors des coefficients $a,$ $b,$ $c,$ $A,$ $B,$ $C$ tels que le champ $\mathrm{X}$ correspondant s'annule en les $m_i.$ Si $m_i=(x_i,y_i)$ on doit r\'esoudre le syst\`eme lin\'eaire
\begin{equation}\label{syslin}
\left\{\begin{array}{ll}
(x_i^2-y_i)a+(x_iy_i-y_i)b+(x_i-y_i)c=y_i-x_i^2y_i\\
(y_i^2-y_i)A+(x_iy_i-y_i)B+(x_i-y_i)C=y_i-x_iy_i^2
\end{array}\right.
\end{equation}
avec $i=1,$ $2,$ $3.$ Notons $D(m_1,m_2,m_3)$ le 
d\'eterminant de la matrice associ\'ee. Si les points 
$m_i$ sont choisis de sorte
que~$D(m_1,m_2,m_3)$ soit non nul, on peut trouver le
champ $\mathrm{X}$ recherch\'e. Il reste ensuite \`a v\'erifier
que $\mathrm{X}$ d\'efinit bien un feuilletage quadratique
c'est-\`a-dire que les composantes de $\mathrm{X}$ sont sans
facteur commun. On est alors assur\'e que pour ces $7$ points, 
il n'y a pas de conique passant par $6$ d'entre eux; en effet
si $\mathcal{C}$ est une conique lisse et $\mathcal{F}$
un feuilletage quadratique sur le plan projectif complexe, 
$\mathcal{C}$ contient au plus cinq points singuliers de 
$\mathcal{F}$ (\emph{voir} \cite{CO}).
On constate que ces involutions sont bien de degr\'e $8.$

\smallskip

\noindent Avec les notations pr\'ec\'edentes l'\'enonc\'e qui suit 
r\'esume la proc\'edure qui n'utilise que de l'alg\`ebre lin\'eaire.

\begin{pro}
{\sl Soient $m_1,$ $m_2,$ $m_3$ trois points g\'en\'eriques dans $\mathbb{C}^2.$ Soit $(a,b,c,A,B, C)$ une solution de (\ref{syslin}). Alors l'involution associ\'ee aux sept points $(1:0:0),$ $(0:1:0),$ $(0:0:1),$ $(1:1:1),$ $m_1,$ $m_2$ et $m_3$ est donn\'ee par $\left(\frac{U_1}{TU_2},\frac{V_1}{TV_2}\right).$ C'est une involution de Geiser.}
\end{pro}

\begin{rem}
Dans \cite{Will} l'auteur explique comment, lorsque les $m_i$ sont en position sp\'eciale, l'involution d\'eg\'en\`ere (baisse du degr\'e, courbes invariantes...)
\end{rem}

\begin{eg} 
Si $m_1=\left(\frac{1}{2},\frac{3}{4}\right),$ $m_2=\left(\frac{3}{4},
\frac{1}{2}\right)$ et $m_3=\left(\frac{4}{3},\frac{2}{3}\right),$
l'involution associ\'ee $\mathcal{I}_\mathcal{F}$ s'\'ecrit alors $\left(\frac{U_1}{TU_2},\frac{V_1}{TV_2}\right)$
avec
\begin{footnotesize}
$$U_1=1220x^3y^2-1844x^3y+693x^3-2456x^2y^3+3624x^2y^2-1054x^2y-198x^2+1184xy^4-1768xy^3+259xy^2+286xy+144y^4-36y^2-54y^3,$$
$$V_1=976x^4y-792x^4+1988x^3y^2-5456x^3y+3267x^3-3848x^2y^3+5652x^2y^2+242x^2y-2178x^2+2936xy^3-5951xy^2+3146xy+414y^3-396y^2,$$
\begin{align*}
&U_2=xy^2-\frac{47}{18}y^2+\frac{13}{18}xy-x+\frac{17}{9}y, && V_2=-\frac{101}{18}x^2+11x^2y-\frac{221}{18}xy+\frac{44}{9}x+2y,
&&T=-2(36x^2-378xy+198x+396y^2-252y);
\end{align*}
\end{footnotesize}

\noindent c'est une involution de Geiser.
\end{eg}

\medskip

\noindent Dans les paragraphes qui suivent, on s'int\'eresse aux involutions associ\'ees \`a des feuilletages ayant des propri\'et\'es, par exemple de sym\'etries, sp\'eciales.

\bigskip

\subsection{Feuilletage quadratique de Jouanolou}

\noindent Le feuilletage $\mathcal{F}_J$ est d\'ecrit 
dans la carte $z=1$ par la $1$-forme 
$$\omega_J=(x^3-y^2)\frac{\partial}{\partial x}+(x^2y-1)\frac{\partial}{\partial y};$$
cet exemple est d\^u \`a Jouanolou (\cite{J}). Historiquement c'est le 
premier exemple connu de feuilletage sans courbe alg\'ebrique invariante
(\cite{J});
c'est aussi un feuilletage qui n'admet pas d'ensemble minimal non 
trivial si l'on en croit~\cite{CdF}.

\noindent  L'involution associ\'ee $\mathcal{I}_{\mathcal{F}_J}$ est donn\'ee par
\begin{small}
\begin{align*}
& (xy^7+3x^5y^2z-x^8-5x^2y^4z^2+2y^3z^5+x^3yz^4-xz^7: 3xy^5z^2+2x^5z^3-x^7y-5x^2y^2z^4+x^4y^3z+yz^7-y^8: \\
&\hspace{6mm} xy^4z^3-5x^4y^2z^2-y^7z+2x^3y^5+3x^2yz^5-z^8+x^7z).
\end{align*}
\end{small}
\hspace{-0.9mm}Elle est de degr\'e $8;$ par ailleurs
$\mathrm{Ind}(\mathcal{I}_{\mathcal{F}_J})=\mathrm{Sing}(\mathcal{F}_J)=\left\{(\xi^j:\xi^{-2j}:1)
\,\vert\, j=0,\ldots,6\right\}$ o\`u $\xi$ d\'esigne une racine $7$-i\`eme de l'unit\'e.

\noindent On constate que $\mathcal{H}(x,y,z)=2(3x^2y^2z^2-xy^5
-x^5z-yz^5);$ puisque $\mathcal{F}_J$ n'a pas de courbe alg\'ebrique invariante
$\mathrm{Flex}(\mathcal{F}_J)$ co\"incide avec le lieu des z\'eros 
de~$\mathcal{H}$ et c'est aussi $\mathrm{Fix}(\mathcal{I}_{\mathcal{F}_J}).$ Le point $(1:1:1)$ est un point singulier de 
$\mathrm{Flex}(\mathcal{F}_J)$ 
de type point double ordinaire. En faisant agir le groupe d'isotropie
de $\mathcal{F}_J$ 
$$\mathrm{Iso}(\mathcal{F}_J)=\langle (y:z:x),\,(\xi^j x:\xi^{-2j}y:z) \,\vert
\,j=0..6\rangle$$ on constate que chaque point singulier de $\mathcal{F}_J$ est 
un point double ordinaire de $\mathrm{Flex}(\mathcal{F}_J);$ 
on peut v\'erifier que ce sont les seuls points singuliers de 
$\mathrm{Flex}(\mathcal{F}_J).$ 
Le lecteur se convaincra ais\'ement de l'irr\'eductibilit\'e
de $\mathcal{H}.$ Il s'en suit que $\mathrm{Flex}(\mathcal{F}_J)$ est irr\'eductible et de genre $\frac{(6-1)(6-2)}{2}-7=3.$  Les points de $\mathrm{Sing}
(\mathcal{F}_J)$ sont par l'argument pr\'ec\'edent en position g\'en\'erale et par cons\'equent
l'involution $\mathcal{I}_{\mathcal{F}_J}$ est une involution de Geiser. Le groupe engendr\'e par~$\mathcal{I}_{\mathcal{F}_J}$ et $\mathrm{Iso}(\mathcal{F}_J)$ est un groupe fini qui ne peut
\^etre conjugu\'e \`a un sous-groupe de $\mathrm{Aut}(\mathbb{P}^2(\mathbb{C}))$
par une transformation birationnelle (le genre $3$ des points fixes de 
$\mathcal{I}_{\mathcal{F}_J}$ est une obstruction). Ce groupe apparait dans la classification des 
sous-groupes finis de~$\mathrm{Bir}(\mathbb{P}^2(\mathbb{C}))$ (\emph{voir}~\cite{DI}).

\bigskip

\noindent On peut donc \'enoncer le r\'esultat suivant. 

\begin{pro}
{\sl L'involution $\mathcal{I}_{\mathcal{F}_J}$ associ\'ee au feuilletage de Jouanolou de degr\'e $2$ est 
de type Geiser. Le groupe engendr\'e par $\mathcal{I}_{\mathcal{F}_J}$ et le groupe d'isotropie de $\mathcal{F}_J$
est un sous-groupe fini d'ordre~$42,$ non lin\'earisable du groupe de Cremona.} 
\end{pro}

\noindent Si l'involution associ\'ee \`a un feuilletage quadratique est pr\'ecis\'ement
de degr\'e $8,$ est-elle de type Geiser ? La r\'eponse, n\'egative, est donn\'ee au \S \ref{f4}.

\bigskip

\subsection{Feuilletages quadratiques de 
$\mathbb{P}^2(\mathbb{C})$ ayant une unique singularit\'e}

\noindent Contrairement au paragraphe pr\'ec\'edent consacr\'e au 
cas g\'en\'erique, on s'int\'eresse ici \`a une cat\'egorie de feuilletages
quadratiques tr\`es particuliers, ceux qui comptent une seule singularit\'e; la 
classification de ces \'el\'ements 
a \'et\'e \'etablie dans \cite{CDGBM}.

\begin{thm}[\cite{CDGBM}\label{feuilquad}]
{\sl \`A automorphisme de $\mathbb{P}^2(\mathbb{C})$ pr\`es, il y a
quatre feuilletages quadratiques sur~$\mathbb{P}^2(\mathbb{C})$ ayant une
seule singularit\'e; ils sont d\'ecrits par les $1$-formes
suivantes
\begin{align*}
&\omega_1=x^2\mathrm{d}x+y^2(x\mathrm{d}y-y\mathrm{d}x);
&&\omega_2=x^2\mathrm{d}x+(x+y^2)(x\mathrm{d}y-y\mathrm{d}x);\\
&\omega_3=xy\mathrm{d}x+(x^2+y^2)(x\mathrm{d}y-y\mathrm{d}x);
&&\omega_4=(x+y^2-x^2y)\mathrm{d}y+x(x+y^2)\mathrm{d}x.
\end{align*}}
\end{thm}

\noindent On d\'esigne par $\mathcal{F}_k$ le feuilletage associ\'e
\`a $\omega_k.$

\noindent Dans chaque \'eventualit\'e le point singulier est situ\'e
en l'origine: $\mathrm{Sing}(\mathcal{F}_k)=\{(0:0:1)\}.$ Les trois premiers feuilletages 
produisent des involutions de Jonqui\`eres de degr\'e inf\'erieur ou \'egal \`a $4.$ Examinons
par exemple l'involution associ\'ee \`a~$\mathcal{F}_1.$

\bigskip

\subsubsection{Le feuilletage $\mathcal{F}_1$ et son involution $\mathcal{I}_{\mathcal{F}_1}$}

\noindent L'involution associ\'ee \`a $\mathcal{F}_1$ est $\mathcal{I}_{\mathcal{F}_1}=(x^3:-x^2y:x^2z-2y^3);$
on remarque qu'elle pr\'eserve la fibration $y/x=$ cte. 
Elle est de degr\'e trois et ses points fixes $\mathrm{Fix}(\mathcal{I}_{\mathcal{F}_1})$
co\"incident avec la droi\-te~$y=0;$ c'est exactement l'ensemble des 
points d'inflexion de $\mathcal{F}_1.$ L'unique droite contract\'ee par $\mathcal{I}_{\mathcal{F}_1}$
est la droite d'\'equa\-tion~$x=~0.$
La transformation $\mathcal{I}_{\mathcal{F}_1}$ s'\'ecrit $(-y,z-2y^3)$ dans la carte
affine $x=1;$ elle est conjugu\'ee  \`a $(-y,z)$ via~$\ell_1=(y,z+y^3).$

\noindent On sait, d'apr\`es \cite{CDGBM}, que
$\mathrm{Iso}(\mathcal{F}_1)=\{(\beta^3 x:\beta^2y:z+\gamma x)\,
\vert\, \gamma\in\mathbb{C},\,\beta\in\mathbb{C}^*\};$
ce groupe est isomorphe au groupe des transformations
affines de la droite. On constate que tout \'el\'ement de~$\mathrm{Iso}
(\mathcal{F}_1)$ est invariant par conjugaison par $\ell_1;$
autrement dit le groupe $\langle\mathrm{Iso}(\mathcal{F}_1),\, \mathcal{I}_{\mathcal{F}_1}\rangle$
engendr\'e par~$\mathrm{Iso}(\mathcal{F}_1)$ et $\mathcal{I}_{\mathcal{F}_1}$ est 
lin\'earisable (via $\ell_1$) dans $\mathrm{Bir}(\mathbb{P}^2(\mathbb{C})).$

\bigskip

\subsubsection{Le feuilletage $\mathcal{F}_4$ et l'involution $\mathcal{I}_{\mathcal{F}_4}$}\label{f4}

\noindent Le feuilletage $\mathcal{F}_4$ induit l'involution 
de degr\'e $8$ donn\'ee par
\begin{small}
\begin{align*}
&\mathcal{I}_{\mathcal{F}_4}=((xz+y^2)(xyz+x^3+y^3)^2: ((2x^2-yz)(xyz+x^3+y^3)-x^5+x^3yz-x^2z^3-xy^2z^2)(xyz+x^3+y^3):\\
&\hspace{12mm} xy^7-x^7y-3xy^4z^3-3x^2y^2z^4+4x^4yz^3+6x^2y^5z+
9x^3y^3z^2-x^4y^4+x^5y^2z-x^3z^5+2x^6z^2-y^6z^2).
\end{align*}
\end{small}

\noindent Comme on l'a dit pr\'ec\'edemment un feuilletage quadratique 
g\'en\'erique produit une involution de Geiser de degr\'e $8$ donc. Le feuilletage
$\mathcal{F}_4$ produit une involution de degr\'e $8:$ ainsi la d\'eg\'en\'erescence du feuilletage
ne conduit pas toujours \`a une perte de degr\'e de l'involution associ\'ee.
Toutefois $\mathcal{I}_{\mathcal{F}_4}$ n'est pas de type Geiser. Elle poss\`ede un seul
point d'ind\'etermination et son ensemble exceptionnel $\mathrm{Exc}(\mathcal{I}_{\mathcal{F}_4})
=\{xyz+x^3+y^3=0\}$ est une cubique \`a point double.

\noindent Un calcul montre que 
$\mathcal{H}(x,y,z)=x^4yz+x^3y^3+x^6-3xy^4z-y^6-x^3z^3-3x^2y^2z^2.$
D'apr\`es \cite{CDGBM} le feuilletage $\mathcal{F}_4$ n'admet pas de 
courbe alg\'ebrique invariante, par suite $\mathrm{Flex}(
\mathcal{F}_4)=\mathrm{Fix}(\mathcal{I}_{\mathcal{F}_4})$ et $(\mathcal{H}=0)$ co\"incident.

\noindent Le changement de variables $(y,z-y^2)$ 
permet de constater que $\mathrm{Flex}(\mathcal{F}_4)=\mathrm{Fix}(\mathcal{I}_{\mathcal{F}_4})$ est une courbe 
rationnelle et $\mathcal{I}_{\mathcal{F}_4}$ est de type projectif.

\noindent Le groupe d'isotropie de $\mathcal{F}_4$ est donn\'e par
$\mathrm{Iso}(\mathcal{F}_4)=\{\mathrm{id},\,(\mathbf{j}x:
\mathbf{j}^2y:z),\,(\mathbf{j}^2x:\mathbf{j}y:z)\},$
o\`u $\mathbf{j}=\mathrm{e}^{2\mathbf{i}\pi/3},$ 
et commute, comme on l'a dit, \`a $\mathcal{I}_{\mathcal{F}_4}.$ Le groupe 
engendr\'e par $\mathcal{I}_{\mathcal{F}_4}$ et $\mathrm{Iso}(\mathcal{F}_4)$ est un 
sous-groupe ab\'elien \`a huit \'el\'ements de $\mathrm{Bir}
(\mathbb{P}^2(\mathbb{C})).$

\bigskip

\subsection{Feuilletages et pinceaux de coniques}

\noindent Les feuilletages $\mathcal{F}$ associ\'es \`a des pinceaux de coniques
jouent un r\^ole particulier dans la classification des feuilletages
quadratiques car leur ensemble $\mathrm{Flex}$ est vide; les z\'eros
de $\mathcal{H}$ sont donc des droites invariantes (six pour un 
pinceau g\'en\'erique). L'exemple qui suit montre que ces droites 
peuvent jouer pour l'involution associ\'ee $\mathcal{I}_\mathcal{F}$ des r\^oles diff\'erents et aussi qu'\`a l'inverse du cas g\'en\'erique
$\mathrm{Flex}(\mathcal{F})$ et~$\mathrm{Fix}(\mathcal{I}_\mathcal{F})$
peuvent \^etre diff\'erents. 

\noindent On consid\`ere le feuilletage $\mathcal{F}_5$ d\'efini par le
pinceau $\frac{y^2-xz}{xy}=$ cte. L'involution $\mathcal{I}_{\mathcal{F}_5}$
associ\'ee est donn\'ee par $\mathcal{I}_{\mathcal{F}_5}=(-x^2:xy:xz+~2y^2).$
 On constate que la droite $y=0,$ qui est invariante par 
$\mathcal{F}_5,$ est une droite de points fixes de~$\mathcal{I}_{\mathcal{F}_5}.$ 
Par contre la droite $x=0,$ qui est elle aussi invariante, 
est contract\'ee par $\mathcal{I}_{\mathcal{F}_5}$ sur le point $(0:0:1).$

\noindent On remarque que $$\mathrm{Iso}(\mathcal{F}_5)=\left\{(\gamma^2 x,\gamma y),\,
 \left(\frac{x}{1+\beta y},\frac{y}{1+\beta y}\right)\,
\Big\vert\, \gamma\in\mathbb{C}^*,\,\beta\in\mathbb{C}\right\}.$$
En se pla\c{c}ant dans la carte $x=1$ on montre que 
$\mathcal{I}_{\mathcal{F}_5}$ est conjugu\'ee \`a $(x:-y:-z)$ par la 
transforma\-tion~$(x^2:-yx:xz-y^2).$ Il se trouve que cette application
commute \`a $\mathrm{Iso}(\mathcal{F}_5);$
le groupe engendr\'e par $\mathrm{Iso}(\mathcal{F}_5)$ et $\mathcal{I}_{\mathcal{F}_5}$
est ici lin\'earisable.

\bigskip

\section{Feuilletages de degr\'e sup\'erieur et involutions}

\medskip

\noindent Soit $\mathcal{I}$ une involution birationnelle que l'on \'ecrit $\mathcal{I}=(\mathcal{I}_1,\mathcal{I}_2)$
dans la carte affine $z=1.$ On suit une d\'emarche inverse \`a la
pr\'ec\'edente: on associe \`a $\mathcal{I}$ le feuilletage $\mathcal{F}$
d\'efini par le champ de vecteurs rationnel
$$\Big(x-\mathcal{I}_1(x,y)\Big)\frac{\partial}{\partial x}+\Big(y-\mathcal{I}_2(x,y)\Big)\frac{\partial}{\partial y}.$$
Soient $\nu$ le degr\'e de $\mathcal{F}$ et $\mathcal{D}$
une droite g\'en\'erique; $\mathcal{F}$ a, par d\'efinition du degr\'e, $\nu$ 
points de tangence avec $\mathcal{D}.$ Si $q$ est l'un de ces points, alors 
par construction son image $\mathcal{I}(q)$ par~$\mathcal{I}$ est aussi un point de tangence; 
par suite $\nu$ est pair.

\noindent Les 
degr\'es de $\mathcal{I}$ et $\mathcal{F}$ sont reli\'es comme suit:
$\deg\mathcal{F}\leq\deg \mathcal{I}-\deg\mathrm{Fix}(\mathcal{I}).$
Si $\mathrm{Fix}(\mathcal{I})$ est irr\'eductible, $\deg \mathcal{I}-\deg\mathcal{F}$
est un multiple de $\deg\mathrm{Fix}(\mathcal{I}).$ On en d\'eduit par exemple que le 
feuilletage associ\'e \`a une involution de Bertini est de de\-gr\'e~$17-9=8.$
Bien s\^ur si $\mathcal{I}$ est une involution de Geiser pr\'ecis\'ement de degr\'e $8$
on retombe sur la construction associ\'ee aux feuilletages quadratiques.
On observe donc, lorsque le degr\'e de $\mathcal{F},$ not\'e $2n,$
est strictement plus grand que $2,$ le ph\'enom\`ene suivant: chaque droite
g\'en\'erique contient $n$ orbites distinctes suivant $\mathcal{I}.$ Il en est ainsi pour les involutions de
Bertini o\`u les orbites sont arrang\'ees en constellations de $4$ orbites en 
alignement.

\bigskip

\noindent Le degr\'e du feuilletage associ\'e  \`a l'involution de Cremona
$\sigma=\left(\frac{1}{x},\frac{1}{y}\right)$ est aussi $2,$ ceci \'etant li\'e au fait que
les points fixes de $\sigma$ sont isol\'es. 

\noindent Il ne faut pas croire que le 
feuilletage associ\'e d\'etermine l'involution comme le montrent les exemples
qui suivent; soit~$R\in\mathbb{C}(x)$ une fonction rationnelle de degr\'e
quelconque. Le feuilletage associ\'e \`a l'involution $\left(x,\frac{R(x)}{y}\right)$
est le pinceau de droites $x=$ cte et ceci ind\'ependamment du choix de 
$R.$ Toutefois il s'agit de cas exceptionnels.

\bigskip

\noindent L'involution de type de Jonqui\`eres
$$\mathcal{I}=\left(\frac{y}{1+x^2y^2},x(1+x^2y^2)\right)$$
est de degr\'e $9.$ Elle pr\'eserve la fibration $xy=$~cte fibre \`a fibre.
Les points fixes de $\mathcal{I}$ sont donn\'es par
$\mathrm{Fix}(\mathcal{I})=\{y-x-x^3y^2=0\}.$ Le lieu d'ind\'etermination de $\mathcal{I}$ est constitu\'e des points $(1:0:0)$ et $(0:1:0)$
et son lieu exceptionnel des courbes $z=0$ et $z^4+x^2y^2=0.$

\noindent Le feuilletage $\mathcal{F}$ associ\'e \`a cette involution est de degr\'e $4$ et d\'ecrit par le champ
$$-\frac{\partial}{\partial x}+(1+x^2y^2)\frac{\partial}{\partial y}.$$
Il d\'efinit un feuilletage de Ricatti sur $\mathbb{P}^1(\mathbb{C})\times
\mathbb{P}^1(\mathbb{C}):$ il est en effet transverse \`a la fibration $x=$ cte. 
L'\'equation diff\'erentielle induite, apr\`es r\'eduction \`a une \'equation lin\'eaire
du second ordre, s'int\`egre via les fonctions de Bessel.

\noindent Si $\mathcal{D}$ est une droite g\'en\'erale les quatre points de
tangences de $\mathcal{D}$ et $\mathcal{F}$ constituent deux 
orbites de~$\mathcal{I}.$

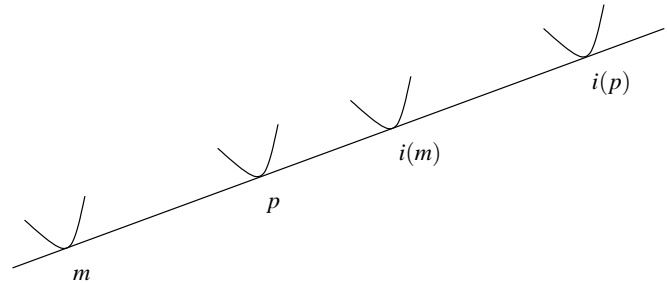
\begin{figure}[H]
\begin{center}
\input{inv.pstex_t}
\end{center}
\end{figure}

\noindent Avec les notations habituelles on obtient 
$\mathcal{H}(x,y,z)=2xyz^5(xz^4-yz^4+x^3y^2).$
Les droites $x=0$ et $y=0$ ne sont pas invariantes par le feuilletage, par 
contre $z=0$ l'est. Par suite $\mathrm{Flex}(\mathcal{F})$ est l'union des courbes 
$x=0,$ $y=0$ et $x-y-x^3y^2=0;$ cette derni\`ere courbe, qui est pr\'ecis\'ement 
l'ensemble $\mathrm{Fix}(\mathcal{I}),$ est elliptique. On constate que les courbes de points
d'inflexion~$x=0$ et $y=0$ ne font pas partie des points fixes de notre involution~$\mathcal{I}.$ En voici l'explication: les courbes~$x=0$ et $y=0$ sont en fait \'echang\'ees par 
$\mathcal{I}.$

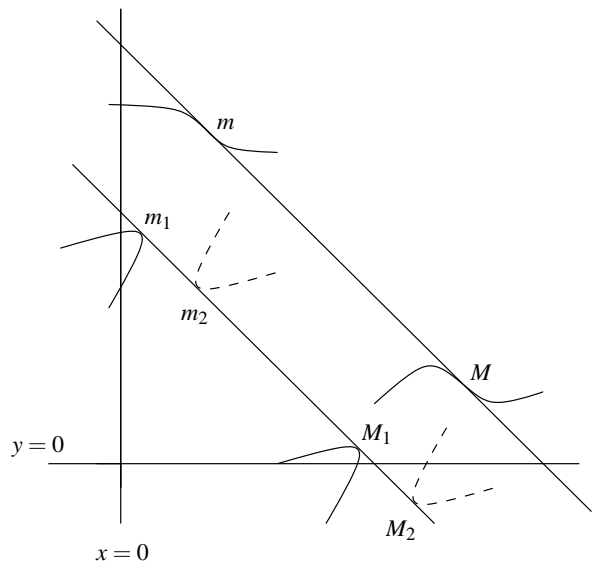
\begin{figure}[H]
\begin{center}
\input{bifurc.pstex_t}
\end{center}
\end{figure}

\noindent On a bifurcation de deux paires de points de contacts simples $m_1,$ 
$m_2$ et $M_1,$ $M_2$ vers les points d'inflexion $m$ et $M;$ \`a 
l'inverse du cas quadratique, l'involution n'\'echange pas $m_1$ et $m_2$
(respectivement $M_1$ et $M_2$), mais disons $m_1$ et $M_1$ (respectivement
$m_2$ et $M_2$).

\bigskip

\section{Feuilletages de degr\'e $3$ et trivolutions}

\medskip

\subsection{Classification des trivolutions birationnelles}

\noindent Les transformations birationnelles p\'eriodiques de p\'eriode $3,$
encore appel\'ees trivolutions birationnelles, ont
\'et\'e class\'ees dans \cite{dF}: une trivolution birationnelle $\mathcal{T}$ 
est \`a conjugaison birationnelle pr\`es 
\begin{itemize}
\item ou bien un automorphisme de $\mathbb{P}^2(\mathbb{C});$

\item ou bien une trivolution de Jonqui\`eres, {\it i.e.} une transformation birationnelle d'ordre $3$ qui pr\'eserve une fibration rationnelle;

\item ou bien un automorphisme sur une surface de del Pezzo de degr\'e $3$ d\'efinie par une \'equation de la for\-me~$x^3=~F(y,z,w)$ dans $\mathbb{P}^3(\mathbb{C})$ avec $F$ polyn\^ome homog\`ene de degr\'e $3$ \`a singularit\'e isol\'ee; dans ce cas $\mathcal{T}$ est la restriction de l'automorphisme de $\mathbb{P}^3(\mathbb{C})$ d\'efini par $(\xi x:y:z:w),$ avec~$\xi^3=1,$ 
(la courbe de points fixes de $\mathcal{T}$ est alors isomorphe \`a la courbe elliptique $\Gamma=\{(y:z:w)\in\mathbb{P}^2(\mathbb{C})\,\vert\,F(y,z,w)=0\});$

\item ou bien un automorphisme sur une surface de del Pezzo de degr\'e $6$ d\'ecrite par $z^3=w^2+F_6(x,y,w)$ (avec $F_6$ polyn\^ome homog\`ene de degr\'e $6$) dans l'espace projectif \`a poids $\mathbb{P}(1,1,2,3)$ muni des coordonn\'ees $(x,y,z,w);$ alors $\mathcal{T}$ est la restriction de l'automorphisme de $\mathbb{P}(1,1,2,3)$ donn\'e par $(x:y:\xi z:w),$ avec $\xi^3=1.$  
\end{itemize}

\bigskip

\subsection{Construction de trivolutions et premi\`eres propri\'et\'es}\label{construction}

\noindent Soit $\mathcal{F}$ un feuilletage de degr\'e~$3$ sur
$\mathbb{P}^2(\mathbb{C});$ toute 
droite g\'en\'erique de $\mathbb{P}^2(\mathbb{C})$ est tangente 
\`a $\mathcal{F}$ en trois points. L'\og application\fg\,
 qui \'echange ces trois points est en g\'en\'eral multivalu\'ee.

\noindent On cherche \`a construire des feuilletages de degr\'e $3$ tels
que l'\og application\fg\, associ\'ee soit une trivolution birationnelle~$\mathcal{T}_\mathcal{F}.$
Notons que pour une telle transformation $\mathcal{T}_\mathcal{F}$ la propri\'et\'e 
suivante est v\'erifi\'ee: pour tout point $p$ de $\mathbb{P}^2(\mathbb{C})\setminus\mathrm{Ind}(
\mathcal{T}_\mathcal{F}),$ les points $p,$ $\mathcal{T}_\mathcal{F}(p)$
et $\mathcal{T}_\mathcal{F}^2(p)$ sont align\'es. 

\begin{rem}
Soit $f\colon\mathbb{P}^2(\mathbb{C})\dashrightarrow\mathbb{P}^2(\mathbb{C})$ une transformation rationnelle non d\'eg\'en\'er\'ee, {\it i.e.} g\'en\'eriquement
dominante,
telle que pour tout $m$ dans~$\mathbb{P}^2(\mathbb{C})$ les points $m,$ $f(m)$ et $f^2(m)$ 
soient align\'es; alors 
\begin{itemize}
\item ou bien $f$ pr\'eserve une fibration en droites fibre \`a fibre;

\item ou bien $f$ est birationnelle p\'eriodique. 
\end{itemize}

\noindent En effet supposons que $f$ ne soit pas p\'eriodique. Un argument de Baire montre que si $m$ est un point g\'en\'erique de~$\mathbb{P}^2(\mathbb{C})$ 
l'orbite positive $\{m,\,f(m),\,f^2(m),\,\ldots\}$ de $m$ sous l'action de $f$ est infinie; alors la droite passant par $m$ et $f(m)$ est 
invariante par $f.$ Il s'en suit que $f$ pr\'eserve une infinit\'e de droites et donc une fibration en droites fibre \`a fibre.

\noindent Consid\'erons par exemple les transformations 
$f_{\kappa,n}=(\kappa x+y^{3n},\kappa y)$
o\`u $\kappa$ est une racine cubique de l'unit\'e et $n$ un entier relatif non nul. 
On v\'erifie que les $f_{\kappa,n}$ satisfont tous la condition d'alignement 
$$\det(f_{\kappa,n}(m)-m,f^2_{\kappa,n}(m)-m)\equiv 0.$$
Pour $\kappa=1$ les transformations $f_{1,n}$ sont non p\'eriodiques et laissent
la fibration $y=$ cte invariante. Pour $\kappa=\mathbf{j}$ ou $\mathbf{j}^2$ les~$f_{\kappa,n}$ sont p\'eriodiques de p\'eriode $3.$
Ces derni\`eres transformations ne laissent pas de fibration en droites invariante fibre \`a fibre.
\end{rem}

\noindent Soient $\mathcal{F}$ un feuilletage de degr\'e $3$ sur $\mathbb{P}^2(\mathbb{C})$ et $\mathrm{X}=\mathrm{X}_1\frac{\partial}{\partial x}+
\mathrm{X}_2\frac{\partial}{\partial y}$ un champ de vecteurs polynomial le d\'efinissant dans la carte affine $(x,y).$ Soit $m=(x,y)$ un point de~$\mathbb{C}^2$ non singulier 
pour $\mathrm{X}.$ Si $t$ est un nombre complexe, on note $\mathrm{Y}_1(t)$
et~$\mathrm{Y}_2(t)$ les composantes du champ $\mathrm{X}$
\'evalu\'ees au point $m+t\mathrm{X}(m).$ Consid\'erons le polyn\^ome $Q(t)\in\mathbb{C}[x,y][t]$
donn\'e par $$Q(t)=\mathrm{Y}_1(t)\mathrm{X}_2-\mathrm{Y}_2(t)\mathrm{X}_1.$$
Les trois racines $t_i$ de 
$Q$ produisent les points $m+t_i\mathrm{X}(m)$ o\`u le feuilletage est 
tangent \`a la droite param\'etr\'ee par $t\mapsto m+t\mathrm{X}(m).$ Le 
polyn\^ome $Q$ \'etant divisible par~$t,$ il s'\'ecrit $tP(t)$ avec 
$P(t)=at^2+bt+c$ et $a,$ $b,$ $c$ dans $\mathbb{C}[x,y].$

\noindent On note~$\Delta(P)$ (resp. 
$\Delta(Q)$) le discriminant de~$P$ (resp. $Q$). On v\'erifie que 
$\Delta(Q)=c^2\Delta(P).$ 

\noindent Avec les notations pr\'ec\'edentes on a l'\'enonc\'e suivant.

\begin{pro}\label{carreee}
{\sl Supposons que le discriminant $\Delta(P)=b^2-4ac\in\mathbb{C}[x,y]$ soit
un carr\'e $s^2$ dans $\mathbb{C}[x,y];$ notons $r_1,$~$r_2$ (dans $\mathbb{C}(x,y)$) les racines de $P.$

\noindent Les applications 
$\mathcal{T}_i=(x+r_i\mathrm{X}_1,y+r_i\mathrm{X}_2)$ sont birationnelles,
p\'eriodiques de p\'eriode trois; on a $\mathcal{T}_1\circ\mathcal{T}_2=\mathrm{id}.$}
\end{pro}

\begin{proof}[{\sl D\'emonstration}]
Les transformations $\mathcal{T}_1$ et $\mathcal{T}_2$ sont donn\'ees \`a r\'eindexation pr\`es par $$\mathcal{T}_1(x,y)=\left(x+\frac{-b+s}{2a}\mathrm{X}_1,y+\frac{-b+s}{2a}\mathrm{X}_2\right),$$ respectivement $$\mathcal{T}_2(x,y)=\left(x-\frac{b+s}{2a}\mathrm{X}_1,y-\frac{b+s}{2a} \mathrm{X}_2\right).$$ Elles sont rationnelles. Les trois points $(x,y),$ $\mathcal{T}_1(x,y),$ $\mathcal{T}_2(x,y)$ repr\'esentent les trois points de tangence du feuilletage~$\mathcal{F}$ avec la droite $(x,y)+\mathbb{C}\cdot\mathrm{X}$ et g\'en\'eriquement ces trois points sont distincts. En particulier $\mathcal{T}_2(\mathcal{T}_1(x,y))$ est l'un de ces trois points de contact; mais il est diff\'erent (g\'en\'eriquement) de $\mathcal{T}_1(x,y),$ sinon $\mathcal{T}_2$ serait l'identit\'e, et de fa\c{c}on analogue diff\'erent de $\mathcal{T}_1.$ Par suite $\mathcal{T}_2\circ\mathcal{T}_1=\mathrm{id}$ et les $\mathcal{T}_i$ sont birationnelles p\'eriodiques de p\'eriode $3.$
\end{proof}

\begin{rem}
Notons que si l'on change $\mathrm{X}$ en $h\cdot\mathrm{X},$ avec $h$ rationnel, la construction produit la m\^eme trivolution (remplacer $\mathbb{C}[x,y]$ par $\mathbb{C}(x,y)$ dans la Proposition \ref{carreee}): la construction ne d\'epend que du feuilletage et non du champ le d\'efinissant.
\end{rem}

\begin{rem}
Pour un feuilletage $\mathcal{F}$ de degr\'e $3$ quelconque on peut associer les transformations multivalu\'ees $\left(x+\frac{-b\pm\sqrt{\Delta}}{2a}\mathrm{X}_1,y+\frac{-b\pm\sqrt{\Delta}}{2a}\mathrm{X}_2\right).$ Visiblement on pourra associer une trivolution \`a $\mathcal{F}$ si et seulement si le discriminant est un carr\'e. 
\end{rem}

\begin{rem}
On peut d'autre part calculer 
explicitement les coefficients $a,$ $b$ et $c.$ Par exemple 
on a en carte affine
$$c=\mathcal{H}(x,y,1)=\mathrm{det}\left[
\begin{array}{ccc}
 x & \mathrm{X}_1 & \mathrm{X}_1\frac{\partial\mathrm{X}_1}{\partial x}+\mathrm{X}_2\frac{\partial
\mathrm{X}_1}{\partial y}\\
 y & \mathrm{X}_2 & \mathrm{X}_1\frac{\partial\mathrm{X}_2}{\partial 
x}+\mathrm{X}_2\frac{\partial\mathrm{X}_2}{\partial y}\\
 1 & 0 & 0
\end{array}
\right].$$
\end{rem}

\begin{rem}
\noindent Il est facile de v\'erifier qu'en g\'en\'eral un feuilletage de degr\'e $3$ ne produit pas
de trivolution. Pour cel\`a consid\'erons le feuilletage de Jouanolou donn\'e par le champ
\begin{align*}
(y^3-x^4)\frac{\partial}{\partial x}+(1-x^3y)\frac{\partial}{\partial y}.
\end{align*}

\noindent Le polyn\^ome
\begin{small}
\begin{align*}
&\Delta(P)=-3\Big(x^{20}-10x^{16}y^3+4x^{15}y^7+10x^{13}y^2+15x^{12}y^6-10x^{11}y^{10}-10x^{10}y-10x^9y^5-10x^8y^9+(10y^{13}+4)x^7+\\
&\hspace{2cm}15x^6y^4
-10x^5y^8+15x^4y^{12}-10(y^3+y^{16})x^3+(y^{20}+10y^7)x^2-10xy^{11}+y^2+4y^{15}\Big).
\end{align*}
\end{small}
\hspace{-0.2mm} n'est pas un carr\'e car sa 
restriction \`a $x=0$ n'en est pas un. 
\end{rem}

\begin{defi}
Un {\sl $d$-tissu singulier (local)} est la donn\'ee d'une famille $\{\mathcal{F}_1,\mathcal{F}_2,\ldots,
\mathcal{F}_d\}$ de feuilletages holomorphes singuliers tels qu'en tout point 
g\'en\'erique $m$ les feuilletages $\mathcal{F}_i$ soient deux \`a deux transverses.\end{defi}

\begin{rem}
Soit $\mathcal{T}$ une trivolution birationnelle; supposons que g\'en\'eriquement les points $m,$ $\mathcal{T}(m)$ et $\mathcal{T}^2(m)$ soient non align\'es. On peut associer \`a $\mathcal{T}$ un $2$-tissu de la fa\c{c}on suivante: les tangentes aux deux feuilles locales passant par $m$ sont les deux droites joignant $m$ \`a $\mathcal{T}(m)$ et $m$ \`a $\mathcal{T}^2(m).$
\end{rem}

\subsection{Premiers exemples}

\noindent Donnons des exemples de feuilletages $\mathcal{F}$ auxquels on peut associer une trivolution; pour
cel\`a on force les polyn\^omes $\Delta(P)$ \`a \^etre des carr\'es. Pour tous ces exemples de trivolutions 
les orbites sont form\'ees de trois points align\'es, ce qui n'est \'evidemment pas le cas g\'en\'eral pour une 
trivolution.

\begin{eg}\label{cegal0}
Consid\'erons le feuilletage $\mathcal{F}$ d\'efini par 
le champ de vecteurs polynomial
$$x^3\frac{\partial}{\partial x}+\frac{\partial}{\partial y}.$$
Ses feuilles sont les courbes rationnelles $y+\frac{1}{2x^2}=$ cte.
 On constate que $\mathcal{H}(x,y,z)=-3x^5z^4;$
on en d\'eduit que $\mathrm{Flex}(\mathcal{F})=~\emptyset.$ 

\noindent On remarque que $\Delta(P)=-3x^{14}$ est effectivement un carr\'e. 
La trivolution associ\'ee \`a $\mathcal{F}$ est 
\begin{align*}
&\mathcal{T}_\mathcal{F}=\left(\mathbf{j}x,y+\frac{\mathbf{j}-1}{x^2}\right), &&
\mathrm{j}^3=1;
\end{align*}
\hspace{-0.2mm}elle laisse la fibration $x=$ cte invariante et est de degr\'e $3.$
On a \begin{align*}
&\mathrm{Ind}(\mathcal{T}_\mathcal{F})=\{(0:1:0)\}, && \mathrm{Exc}(\mathcal{T}_\mathcal{F})=
\{x=0\}, &&\mathrm{Fix}(\mathcal{T}_\mathcal{F})=\{(1:0:0)\}.
\end{align*}

\noindent La transformation $\mathcal{T}_\mathcal{F}^2$ s'\'ecrit
$$\left(\mathbf{j}^2x,y+\frac{2\mathbf{j}+1}{\mathbf{j}^2x^2}\right);$$
elle a m\^eme lieux d'ind\'etermination, exceptionnel
et ensemble de points fixes que $\mathcal{T}_\mathcal{F}.$ 
Finalement $\mathrm{Flex}(\mathcal{F})$ est vide et~$\mathrm{Fix}(t_
\mathcal{F})=\mathrm{Fix}(\mathcal{T}_\mathcal{F}^2)$ se r\'eduit \`a un point.

\noindent On peut v\'erifier que 
$$\mathrm{Iso}(\mathcal{F})=\{(x,y+\alpha),\,(\beta x,y/\beta^2)\,\vert\,\beta\in\mathbb{C}^*,\,\alpha\in\mathbb{C}\}.$$
Remarquons que $\mathcal{T}_\mathcal{F}$ commute \`a 
$\mathrm{Iso}(\mathcal{F});$
le groupe engendr\'e par $\mathrm{Iso}(\mathcal{F})$ et $\mathcal{T}_\mathcal{F}$ est une extension 
triple du groupe affine de la droite.

\noindent Notons $f_0=y+\frac{1}{2x^2}$ une int\'egrale premi\`ere de 
$\mathcal{F}.$ Posons $f_1=f_0\mathcal{T}_\mathcal{F}$ et $f_2=f_0\mathcal{T}_\mathcal{F}^2$
les transform\'es de $f_0$ par $\mathcal{T}_\mathcal{F}$ et $\mathcal{T}_\mathcal{F}^2.$

\noindent Soit $\mathcal{F}_i$ le feuilletage d\'efini par les niveaux de la fonction rationnelle $f_i.$ Consid\'erons le $3$-tissu $\mathcal{W}=(\mathcal{F}_0,\mathcal{F}_1,\mathcal{F}_2).$ Si~$(a_0,a_1,a_2)$ est une solution non triviale du syst\`eme lin\'eaire 
\begin{align*}
&a_0+a_1+a_2=0, && a_0\mathbf{j}^2+(3-2\mathbf{j}^2)a_1+(5\mathbf{j}+2)a_2=0,
\end{align*}
\hspace{-0.2mm}alors $a_0f_0+a_1f_1+a_2f_2=0.$ En d'autres termes le tissu $\mathcal{W}=(
\mathcal{F}_0,\mathcal{F}_1,\mathcal{F}_2)$ est hexagonal (\cite{Be}). C'est un ph\'enom\`ene que nous allons retrouver de fa\c{c}on r\'ecurrente dans les exemples qui suivent. Rappelons ce 
qu'est un tissu hexagonal. Soit~$m$ un point g\'en\'erique; en $m$ le feuilletage
$\mathcal{F}_i$ est d\'efini par les niveaux d'une certaine submersion locale. Lorsque l'on 
peut trouver trois telles submersions $f_i$ d\'efinissant $\mathcal{F}_i$ et satisfaisant
la relation (ab\'elienne) $f_1+f_2+f_3=0$ on dit que le tissu est {\sl hexagonal}.
On renvoie \`a \cite{Be} pour la justification de la terminologie.
\end{eg}

\begin{eg}\label{ref}
Si $\mathcal{F}$ est le feuilletage d\'ecrit par 
$$(x^3-1)\frac{\partial}{\partial x}+\frac{\partial}{\partial y},$$
le polyn\^ome $\Delta(P)$ vaut $-3x^2(x^3-1)^4;$ comme
c'est un carr\'e, on peut associer \`a $\mathcal{F}$ une trivolution 
$$\mathcal{T}_\mathcal{F}=\left(\mathbf{j}x,\frac{x^3y-y+(\mathbf{j}-1)x}{x^3-1}\right)$$ 
pour laquelle on a
\begin{small}
\begin{align*}
&\mathrm{Ind}(\mathcal{T}_\mathcal{F})=\{(0:1:0)\}, &&\mathrm{Exc}(\mathcal{T}_\mathcal{F})=\{x-z=0\}\cup
\{x-\mathbf{j}z=0\}\cup
\{x-\mathbf{j}^2z=0\},
&&\mathrm{Fix}(\mathcal{T}_\mathcal{F})=\{(1:0:0)\}\cup\{x=0\}.
\end{align*}
\end{small}

\noindent Son carr\'e s'\'ecrit 
$$\mathcal{T}_\mathcal{F}^2=\left(\mathbf{j}^2x,\frac{x^3y-y+(\mathbf{j}^2-1)x}{x^3-1}\right)$$
et $\mathrm{Ind}(\mathcal{T}_\mathcal{F}^2)=\mathrm{Ind}(\mathcal{T}_\mathcal{F}),$ 
$\mathrm{Exc}(\mathcal{T}_\mathcal{F}^2)=\mathrm{Exc}(\mathcal{T}_\mathcal{F}),$
$\mathrm{Fix}(\mathcal{T}_\mathcal{F}^2)=\mathrm{Fix}(\mathcal{T}_\mathcal{F}).$

\noindent Un calcul montre que 
$\mathcal{H}(x,y,z)=-3x^2z^4(x-z)(x-\mathbf{j}z)(x-\mathbf{j}^2z).$
Les droites $z=0$ et $x=z$ sont invariantes par $\mathcal{F}$ par 
contre celle d'\'equation $x=0$ (resp. $x-\mathbf{j}z=0,$ resp. $x-\mathbf{j}^2z=0$) ne l'est pas; par suite $\mathrm{Flex}(\mathcal{F})=
\{x=0\}\cup\{x-\mathbf{j}z=0\}\cup\{x-\mathbf{j}^2z=0\}.$ Les droites $x-\mathbf{j}z=0$ 
et $x-\mathbf{j}^2z=0$ sont contract\'ees par $\mathcal{T}_\mathcal{F}$ et $\mathcal{T}_\mathcal{F}^2,$
la droite $x=0$ est l'unique courbe de points fixes de $\mathcal{T}_\mathcal{F}$ (resp. 
$\mathcal{T}_\mathcal{F}^2$); autrement dit les courbes de $\mathrm{Flex}(\mathcal{F})$
sont soit contract\'ees par $\mathcal{T}_\mathcal{F},$ soit fix\'ees par $\mathcal{T}_\mathcal{F}.$

\noindent Alors que le feuilletage pr\'ec\'edent poss\`ede une int\'egrale premi\`ere
rationnelle celui-ci a une int\'egrale premi\`ere de type Liouville
$$3y-\ln(x-1)-\mathbf{j}\ln(x-\mathbf{j})-\mathbf{j}^2\ln(x-\mathbf{j}^2).$$

\noindent Notons $f_0$ cette int\'egrale premi\`ere et posons 
$f_1=f_0\mathcal{T}_\mathcal{F},$ $f_2=f_0\mathcal{T}_\mathcal{F}^2.$
%
\noindent Comme pr\'ec\'edemment le $3$-tissu $\mathcal{W}=(\mathcal{F}_0,\mathcal{F}_1,\mathcal{F}_2)$ (o\`u les $\mathcal{F}_i$
sont les feuilletages par les niveaux des $f_i$) 
est hexagonal. 

\begin{rem}
Le feuilletage associ\'e au champ de vecteurs
$$\mathrm{X}_\varepsilon=(x^3-\varepsilon)\frac{\partial}{\partial x}+\frac{\partial}{\partial y}$$
est conjugu\'e pour $\varepsilon$ non nul au champ $\mathrm{X}_1$ 
consid\'er\'e dans l'exemple \ref{ref} par la transformation lin\'eaire $(\varepsilon^{1/3}x,\varepsilon^{-2/3}y).$
Lorsque $\varepsilon$ tend vers $0,$ le feuilletage et son involution 
$$\left(\mathbf{j}x,\frac{x^3y-\varepsilon y+(\mathbf{j}-1)x}{x^3-\varepsilon}\right)$$
d\'eg\'en\`erent sur ceux de l'exemple \ref{cegal0}; on note une chute
de degr\'e de la trivolution pour $\varepsilon=0$.
\end{rem}
\end{eg}

\begin{eg}\label{degre4}
Consid\'erons le feuilletage 
$\mathcal{F}$ d'int\'egrale premi\`ere~$y-\frac{1}{3}\ln x+\frac{1}{x}+\frac{1}{2x^2};$ il est
aussi d\'efini par $$x^3\frac{\partial}{\partial x}+\left(1+x+\frac{1}{3}x^2\right)\frac{\partial}{\partial y}.$$
Comme $\Delta(P)=-\frac{1}{3}x^{14}(x+3)^2,$ on associe \`a $\mathcal{F}$ la trivolution 
$$\mathcal{T}_\mathcal{F}=\left(\frac{3\mathbf{j}x}{(1-\mathbf{j})x+3},
\frac{3x^2y-x^2+(\mathbf{j}-4)x+3(\mathbf{j}-1)}{3x^2}\right)$$
qui pr\'eserve la fibration $x=$ cte. On a
\begin{align*}
&\mathrm{Ind}(\mathcal{T}_\mathcal{F})=\{(0:1:0),\,(1:0:0)\}, &&\mathrm{Fix}(\mathcal{T}_\mathcal{F})=\{x+3z=0\},
&&\mathrm{Exc}(\mathcal{T}_\mathcal{F})=\{x=0\} \cup\{z=0\}\cup\{(\mathbf{j}-1)x-3z=0\}.
\end{align*}

\noindent On constate que
$$\mathcal{T}_\mathcal{F}^2=\left(\frac{3(\mathbf{j}-1)^4x}{16((\mathbf{j}+2)x+3)},\frac{3(3\mathbf{i}\sqrt{3}
+(4+5\mathbf{j})x+(1+\mathbf{j})x^2-3(1+\mathbf{j})x^2y)
}{(\mathbf{j}-1)^4x^2}\right);$$
on v\'erifie que $\mathrm{Ind}(\mathcal{T}_\mathcal{F}^2)
=\mathrm{Ind}(\mathcal{T}_\mathcal{F}),$ $\mathrm{Fix}(\mathcal{T}_\mathcal{F}^2)=\mathrm{Fix}(\mathcal{T}_\mathcal{F})$
et $\mathrm{Exc}(\mathcal{T}_\mathcal{F}^2)=\{x=0\} \cup\{z=0\}\cup\{(\mathbf{j}+2)x+3z=0\}.$

\noindent \`A l'inverse de l'exemple pr\'ec\'edent les ensembles exceptionnels
de $\mathcal{T}_\mathcal{F}$ et $\mathcal{T}_\mathcal{F}^2$ diff\`erent. Un calcul montre que $
\mathcal{H}(x,y,z)=-\frac{1}{3}x^5z^2(x+3z)^2.$
Les droites $x=0$ et $z=0$ sont invariantes par $\mathcal{F}$ alors
que celle d'\'equation $x+3z=0$ ne l'est pas; par suite 
$\mathrm{Flex}(\mathcal{F})=\{x+3z=0\}$ et $\mathrm{Flex}(\mathcal{F})=
\mathrm{Fix}(\mathcal{T}_\mathcal{F})=\mathrm{Fix}(\mathcal{T}_\mathcal{F}^2).$

\noindent En conjuguant $\mathcal{F}$ par $\left(\varepsilon x,\frac{y}{\varepsilon^3}\right)$ on 
obtient la famille de feuilletages $(\mathcal{F}_\varepsilon)$ d\'ecrite par les champs
\begin{align*}
&x^3\frac{\partial}{\partial x}+\left(1+\varepsilon x+\frac{\varepsilon^2}{3}x^2\right)\frac{\partial}{\partial y},
&& \varepsilon\not=0
\end{align*}
\hspace{-0.2mm}\`a laquelle on peut associer la famille de trivolutions
$$\mathcal{T}_{\mathcal{F}_\varepsilon}=\left(\frac{3\mathbf{j}x}{(1-\mathbf{j})\varepsilon x+3},
\frac{3x^2y-\varepsilon^2x^2+(\mathbf{j}-4)\varepsilon x+3(\mathbf{j}-1)}{3x^2}\right).$$

\noindent La famille $(\mathcal{T}_{\mathcal{F}_\varepsilon})$ de transformations de degr\'e $4$ 
pour $\varepsilon\not=0$ d\'eg\'en\`ere encore pour $\varepsilon=0$ sur la trivolution de degr\'e $3$
trait\'ee dans l'exemple \ref{cegal0}.
\end{eg}

\subsection{Feuilletages homog\`enes}

\noindent On va s'int\'eresser \`a une famille tr\`es sp\'eciale de feuilletages: les 
feuilletages homog\`enes g\'en\'eriques de degr\'e $3.$ \`A conjugaison pr\`es un feuilletage homog\`ene g\'en\'erique de degr\'e $3$ \`a l'origine 
de~$\mathbb{C}^2\subset~\mathbb{P}^2(\mathbb{C})$ est donn\'e par une $1$-forme ferm\'ee rationnelle du type suivant
\begin{align*}
\frac{\mathrm{d}x}{x}+\lambda\frac{\mathrm{d}y}{y}+\mu\frac{\mathrm{d}(y-x)}{y-x}+
\nu\frac{\mathrm{d}(y-\alpha x)}{y-\alpha x} &&\text{avec}
&&\alpha\lambda\mu\nu(1+\lambda+\mu+\nu)\not=0, &&\alpha\not= 1;
\end{align*}
si $\lambda+\mu+\nu+1=0$ le feuilletage correspondant est de degr\'e $0.$
Il poss\`ede l'int\'egrale premi\`ere multivalu\'ee $xy^\lambda(y-x)^\mu
(y-\alpha x)^\nu.$ On note $\mathcal{F}(\alpha;\lambda,\mu,\nu)$ un tel feuilletage.
Il y a donc quatre param\`etres, trois de type r\'esidus $\lambda,$ $\mu,$
$\nu$ et $\alpha$ qui positionne les droites invariantes. 

\begin{defi}
Le quadruplet $(\alpha,\lambda,\mu,\nu)$ est {\sl admissible} si $\alpha\lambda\mu 
\nu(\lambda+\mu+\nu+1)\not=0$ et $\alpha\not=1.$ 
\end{defi}
 
\noindent Le champ de vecteurs homog\`ene 
$$-x\Big(\lambda(y-x)(y-\alpha x)+\mu y(y-\alpha x)+\nu y(y-x)\Big)\frac{\partial}{\partial x}
+y\Big((y-x)(y-\alpha x)-\mu x(y-\alpha x)-\nu\alpha x(y-x)\Big)\frac{\partial}{\partial y}$$
d\'efinit aussi le feuilletage $\mathcal{F}$ car est dans le noyau de la forme ferm\'ee pr\'ec\'edente. 
Un tel $\mathcal{F}$ est invariant par homoth\'etie. En particulier si $\mathcal{F}$ induit une trivolution 
$\mathcal{T}_\mathcal{F},$ celle-ci commute \`a toutes les homoth\'ethies. Un calcul \'el\'ementaire
montre que $\mathcal{T}_\mathcal{F}$ pr\'eserve la fibration radiale $y/x=$ constante: $\mathcal{T}_\mathcal{F}$
est donc de Jonqui\`eres.

\noindent Commen\c{c}ons par quelques exemples simples. 

\begin{eg}\label{exemple}
\noindent Consid\'erons le feuilletage $\mathcal{F}$ donn\'e par 
le champ de vecteurs
$$y^3\frac{\partial}{\partial x}+x^3\frac{\partial}{\partial y}.$$
Les feuilles de $\mathcal{F}$ sont les niveaux de $y^4-x^4$
et sont donc des courbes de genre $3.$ Un calcul montre que 
$$\mathcal{H}(x,y,z)=3x^2y^2z(x+\mathbf{i}y)(x-\mathbf{i}y)(y-x)(y+x).$$
Puisque les droites $x-y=0,$ $x+y=0,$ $x- \mathbf{i}y=0,$ 
$x+\mathbf{i}y=0$ sont invariantes par $\mathcal{F},$
on a~$\mathrm{Flex}(\mathcal{F})=\{x=0\}\cup\{y=0\}\cup\{z=0\}.$

\noindent Le polyn\^ome $\Delta(P)=3x^2y^2(x-y)^4(x+y)^4(x+\mathbf{i}y)^4(x-\mathbf{i}y)^4$ 
\'etant un carr\'e, est associ\'ee \`a $\mathcal{F}$ la trivolution de Jonqui\`eres de degr\'e $5$
donn\'ee par
$$\mathcal{T}_\mathcal{F}=\left(\frac{x(x^4-y^4)}{x^4-\mathbf{j}y^4},
\frac{\mathbf{j}y(x^4-y^4)}{x^4-\mathbf{j}y^4}\right).$$

\noindent L'ensemble d'ind\'etermination de $\mathcal{T}_\mathcal{F}$
$$\mathrm{Ind}(\mathcal{T}_\mathcal{F})=\{(1:1:0),\,(1:-1:0),\,(\mathbf{i}:1:0),\,(-\mathbf{i}:1:0),
\,(0:0:1)\}$$
est le lieu singulier de $\mathcal{F}$ et
$\mathrm{Fix}(\mathcal{T}_\mathcal{F})=\{x=0\}\cup\{y=0\}\cup\{z=0\}$
co\"incide avec l'ensemble des points d'inflexion de~$\mathcal{F}.$ 
L'ensemble exceptionnel de $\mathcal{F}$ 
\begin{small}
$$\mathrm{Exc}(\mathcal{T}_\mathcal{F})=\{x+y=0\}\cup\{x-y=0\}\cup\{
y-\mathbf{i}x=0\}\cup\{y+\mathbf{i}x=0\}\cup\{x-\mathbf{j}y=0\}
\cup\{x+\mathbf{j}y=0\}\cup\{
x-\mathbf{i}\mathbf{j}y=0\}\cup\{x+\mathbf{i}\mathbf{j}y=0\}$$
\end{small}
\noindent est constitu\'e de huit droites.

\noindent On v\'erifie que
$$\mathcal{T}_\mathcal{F}^2=\left(\frac{x(x^4-y^4)}{x^4-\mathbf{j}^2y^4},\frac{\mathbf{j}^2y(x^4-y^4)}{x^4-\mathbf{j}^2y^4}\right)$$
a m\^eme lieu d'ind\'etermination et m\^eme ensemble de points fixes que 
$\mathcal{T}_\mathcal{F}.$ Les ensembles exceptionnels diff\`erent puisque
\begin{small}
$$\mathrm{Exc}(\mathcal{T}_\mathcal{F}^2)=\{x+y=0\}\cup\{x-y=0\}\cup\{
y-\mathbf{i}x=0\}\cup\{y+\mathbf{i}x=0\}
\cup\{x-\mathbf{j}^2y=0\}\cup\{x+\mathbf{j}^2y=0\}
\cup\{x-\mathbf{i}\mathbf{j}^2y=0\}\cup\{x+\mathbf{i}\mathbf{j}^2y=0\}.$$
\end{small}

\noindent Posons $f_0=x^4-y^4.$ Faisons agir $\mathcal{T}_\mathcal{F}$ et $\mathcal{T}_\mathcal{F}^2$ sur $\mathcal{F}$ pour obtenir
les feuilletages $\mathcal{F}_1$ et $\mathcal{F}_2$ d\'efinis par les fonctions rationnel\-les~$f_1=f_0\mathcal{T}_\mathcal{F}$ et $f_2=f_0\mathcal{T}_\mathcal{F}^2.$
\noindent Une fois de plus consid\'erons le $3$-tissu $\mathcal{W}$ associ\'e aux trois fonctions rationnelles $f_0,\,f_1,\,f_2;$ en un point g\'en\'erique il est aussi donn\'e par
\begin{align*}
&F_0=f_0^{-1/3}=\frac{x^4-y^4}{(x^4-y^4)^{4/3}}, && F_1=f_1^{-1/3}=\frac{x^4-\mathbf{j}y^4}{(x^4-y^4)^{4/3}} 
&& \text{et} && F_2=f_2^{-1/3}=\frac{x^4-\mathbf{j}^2y^4}{(x^4-y^4)^{4/3}}.
\end{align*} 

\noindent Soit $(a_0,a_1,a_2)$ une solution non triviale du syst\`eme lin\'eaire 
\begin{align*}
&a_0+a_1+a_2=0, && a_0+\mathbf{j}a_1+\mathbf{j}^2a_2=0;
\end{align*}
\hspace{-0.2mm}alors $a_0F_0+a_1F_1+a_2F_2=0:$ le tissu $\mathcal{W}$ est hexagonal.
\end{eg}

\noindent On obtient ainsi parmi les feuilletages homog\`enes g\'en\'eriques
l'exemple  d'un feuilletage \og hamiltonien\fg\, auquel est associ\'e une 
trivolution. On peut se demander si c'est le seul. La r\'eponse est donn\'ee
par l'\'enonc\'e suivant.

\begin{pro}\label{ununun}
{\sl On peut associer une trivolution \`a $\mathcal{F}=\mathcal{F}(\alpha;1,1,1)$  
si et seulement si $\alpha$ vaut $-1,$ $2$ ou $1/2;$ \`a conjugaison lin\'eaire 
pr\`es on est dans la situation d\'ecrite dans l'exemple \ref{exemple}.}
\end{pro}

\begin{proof}[{\sl D\'emonstration}]
Reprenons les notations introduites au \S \ref{construction}.
L'application associ\'ee \`a $\mathcal{F}$ est birationnelle 
si $\Delta(P)$ est un carr\'e. Or~$\Delta(P)$ s'\'ecrit
$1024x^4y^4(x^2-(1+\alpha)xy+\alpha y^2)^4R(x,y)$
\noindent o\`u $R(x,y)$ d\'esigne
\begin{footnotesize}
\begin{align*}
(1-\alpha+\alpha^2)x^4-4(1-\alpha)^2(1+\alpha)x^3y-2(\alpha^3
+\alpha^2+\alpha-2)x^2y^2-4\alpha(1-\alpha)^2(1+\alpha)xy^3
+\alpha^2(1-\alpha+\alpha^2)y^4;
\end{align*} 
\end{footnotesize}
\noindent donc $\Delta(P)$ est un carr\'e si et seulement si $R$ en est un. Un calcul
montre que $R$ est un carr\'e si et seulement si $\alpha$ prend les valeurs
$-1,$ $1/2$ ou $2.$ 

\noindent Posons $h_\alpha=xy(y-x)(y-\alpha x).$ 
Les polyn\^omes $h_{-1},$ $h_{1/2}$ et $h_2$ sont lin\'eairement conjugu\'es; de plus, 
quitte \`a faire le changement de variables 
\begin{align*}
(\kappa x+\kappa(4\kappa^2-1)y,\kappa x+\kappa(1-4\kappa^2)y), && \text{ o\`u }
\kappa=2^{-3/4}\exp\left(\frac{i\pi}{8}\right),
\end{align*}
\noindent on constate que $y^4-x^4$ s'\'ecrit $xy(x-y)(x+y).$
\end{proof}

\begin{rem}
La configuration des quatre droites qui apparaissent dans la proposition \ref{ununun}
est la configuration sp\'eciale, celle qui poss\`ede un groupe d'automorphismes
\`a $12$ \'el\'ements.
\end{rem}

\noindent La proposition \ref{ununun} poss\`ede de nombreuses g\'en\'eralisations, en voici une. Le sch\'ema de preuve est exactement le m\^eme que celui de la proposition 
\ref{ununun}, les calculs \'etant un petit peu plus p\'enibles.

\begin{pro}
{\sl On consid\`ere $\mathcal{F}$ un \'el\'ement de la famille
$\mathcal{F}(\alpha;1,\mu,\mu).$ On peut associer \`a $\mathcal{F}$ une trivolution $\mathcal{T}_\mathcal{F}$ 
si et seulement si $\mu$ vaut $1$ et $\alpha$ vaut $-1,$ $2$ ou $1/2;$ \`a conjugaison lin\'eaire 
pr\`es on est dans la situation de l'exemple \ref{exemple}.}
\end{pro}

\begin{eg}\label{alpha-1b}
Soit $\mathcal{F}$ le feuilletage donn\'e par 
le champ de vecteurs
$$x^3\frac{\partial}{\partial x}+y^3\frac{\partial}{\partial y}.$$
On remarque que  
$x^2y^2(x-y)^{-1}(x+y)^{-1}$ est une int\'egrale premi\`ere rationnelle de $\mathcal{F};$
ses feuilles sont des quartiques ayant deux singularit\'es ordinaires, donc des 
courbes elliptiques. On v\'erifie que $\mathcal{H}(x,y,z)=3x^3y^3z(x+y)(y-x);$
on en d\'eduit que $\mathrm{Flex}(\mathcal{F})=\emptyset.$
On constate que $\Delta(P)=3x^6y^6(x+y)^4(y-x)^4$ est un carr\'e.
La trivolution associ\'ee \`a $\mathcal{F}$ s'\'ecrit
$$\mathcal{T}_\mathcal{F}=\left(\frac{\mathbf{j}x(x^2-y^2)}{x^2-\mathbf{j}y^2},
\frac{y(x^2-y^2)}{x^2-\mathbf{j}y^2}\right);$$
c'est encore une transformation de Jonqui\`eres qui est de degr\'e $3.$ 
Son ensemble exceptionnel
\begin{align*}
&\mathrm{Exc}(\mathcal{T}_\mathcal{F})=\{x-y=0\}\cup \{x+y=0\}\cup\{\mathbf{j}x-y=0\}
\cup\{\mathbf{j}x+y=0\}
\end{align*}
\hspace{-0.2mm}est constitu\'e de quatre droites dont deux sont invariantes par le feuilletage. On remarque que 
$$\mathcal{T}_\mathcal{F}^2=\left(\frac{x(x^2-y^2)}{\mathbf{j}x^2-y^2},
\frac{\mathbf{j}y(x^2-y^2)}{\mathbf{j}x^2-y^2}\right).$$

\noindent L'ensemble exceptionnel de $\mathcal{T}_\mathcal{F}^2$ 
$$\mathrm{Exc}(\mathcal{T}_\mathcal{F}^2)=\{x-y=0\}\cup \{x+y=0\}\cup\{\mathbf{j}y-x=0\}
\cup\{\mathbf{j}y+x=0\}$$
est encore form\'e des deux m\^emes droites invariantes auxquelles 
s'ajoutent deux nouvelles droites.
C'est un exemple o\`u~$\mathrm{Flex}(\mathcal{F})$ est vide et 
o\`u la trivolution a seulement des points 
fixes isol\'es.

\noindent Posons $f_0=\frac{x^2y^2}{x^2-y^2}$ et $\mathcal{F}_0=\mathcal{F}.$ Faisons agir $\mathcal{T}_\mathcal{F}$ et $\mathcal{T}_\mathcal{F}^2$ sur $\mathcal{F}$ pour obtenir
les feuilletages $\mathcal{F}_1$ et $\mathcal{F}_2$ d\'efinis par les fonctions rationnel\-les~$f_1=f_0\mathcal{T}_\mathcal{F}$ et $f_2=f_0\mathcal{T}_\mathcal{F}^2.$
Consid\'erons le $3$-tissu $\mathcal{W}=(\mathcal{F}_0,\mathcal{F}_1,\mathcal{F}_2);$ en un point g\'en\'erique il est aussi donn\'e par
\begin{align*}
&F_0=f_0^{-1/3}=\frac{x^2-y^2}{(xy(x^2-y^2))^{2/3}}, && F_1=f_1^{-1/3}=
\frac{x^2-\mathbf{j}y^2}{(xy(x^2-y^2))^{2/3}}
&& \text{et} && F_2=f_2^{-1/3}=\frac{x^2-\mathbf{j}^2y^2}{(xy(x^2-y^2))^{2/3}}.
\end{align*} 

\noindent Soient $(a_0,a_1,a_2)$ une solution non triviale du syst\`eme lin\'eaire 
\begin{align*}
&a_0+a_1+a_2=0, && a_0+\mathbf{j}a_1+\mathbf{j}^2a_2=0;
\end{align*}
\hspace{-0.9mm}alors $a_0F_0+a_1F_1+a_2F_2=0:$ le tissu $\mathcal{W}$ est hexagonal. 
\end{eg}

\noindent La proposition \ref{ununun} sugg\`ere que la configuration $\alpha\in\{-1,\,2,\,1/2\}$ joue un r\^ole tr\`es sp\'ecial. 

\noindent Venons-en au cas g\'en\'eral. On va d\'ecrire dans l'espace des param\`etres $\{(\alpha,\lambda,\mu,\nu)\in 
(\mathbb{C}^*)^4\,\vert\,\alpha\not=~1\}$ les feuilleta\-ges~$\mathcal{F}(\alpha;\lambda,\mu,\nu)$ auxquels est associ\'ee une trivolution. 
Introduisons~$\Lambda$ 
le projectivis\'e de l'espace des polyn\^omes homog\`enes de
degr\'e $4$ en $2$ variables; on remarque que~$\Lambda\simeq~\mathbb{P}^4(\mathbb{C}).$ Posons $$\Theta=\{T\in\Lambda\,\vert\, T\text{ est un carr\'e}\}.$$
Un \'el\'ement $T=\tau_1x^4+\tau_2x^3y+\tau_3x^2y^2
+\tau_4xy^3+\tau_5y^4$ appartient \`a $\Theta$ si et seulement s'il existe des com\-plexes~$A,$ $B$ et 
$C$ tels que $T=(Ax^2+Bxy+Cy^2)^2;$ autrement dit si et seulement si 
\begin{align*}
&\tau_1=A^2,&&\tau_2=2AB, && \tau_3=B^2+2AC, &&\tau_4=2BC, &&\tau_5=C^2.
\end{align*}
\hspace{-0.2mm}L'application $\Psi\colon(A,B,C)\mapsto(A^2,2AB,B^2+2AC,2BC,C^2)$
induit une application holomorphe, en fait un plongement de~$\mathbb{P}^2
(\mathbb{C})$ dans $\mathbb{P}^4(\mathbb{C})$ not\'e encore $\Psi.$ Son image, qui est $\Theta,$ 
est en fait une 
projection de la surface de Veronese standard de $\mathbb{P}^5(\mathbb{C})$
sur~$\mathbb{P}^4(\mathbb{C}).$
On peut \'evidemment donner l'id\'eal de d\'efinition de $\Theta.$
Toutefois pour les calculs pratiques il est plus commode d'utiliser la 
pr\'esentation ensembliste suivante. On d\'ecoupe l'espace $\mathbb{P}^4(\mathbb{C})$
comme l'union de l'hyper\-plan~$\tau_5=0$ et de l'ouvert (carte affine) $\tau_5=1$
et nous allons d\'ecrire $\Theta$ relativement \`a ce d\'ecoupage. Pour $T$ comme ci-dessus on note enco\-re~$T=(\tau_1:\ldots:\tau_5).$ Remarquons que $\Psi$ envoie la carte affine $C=1$ de $\mathbb{P}^2(\mathbb{C})$ dans la carte affi\-ne~$\tau_5=1$ de $\mathbb{P}^4(\mathbb{C}).$ Ainsi si $T=(\tau_1:\ldots:\tau_4:0)$ est dans $\Psi(\{C=0\})$ on a 

\noindent\textbf{\textit{(i)}} $\,\,$ $\tau_4=\tau_5=0$ et $4\tau_1\tau_3-\tau_2^2=0.$

\noindent Maintenant si $T=(\tau_1:\ldots:\tau_4:1)$ est dans $\Psi(\{C=1\})$ on v\'erifie que lorsque $\tau_4=0$ on a 

\noindent \textbf{\textit{(ii)}} $\tau_5=1,$ $\tau_2=\tau_4=\tau_3^2-4\tau_1=0$

\noindent et que pour $\tau_4\not=0$ les $\tau_i$ v\'erifient

\noindent \textbf{\textit{(iii)}} $\tau_5=1,$ $\tau_4\not=0,$ $\tau_4^2\tau_1-\tau_2^2=4\tau_3\tau_4-\tau_4^3-8\tau_2=0.$

\noindent On remarque que l'annulation de $\tau_4$ dans \textbf{\textit{(iii)}} ne produit pas \textbf{\textit{(ii)}}.

\bigskip

\noindent Soit donc $\mathcal{F}=\mathcal{F}(\alpha;\lambda,\mu,\nu),$ le quadruplet $(\alpha,\lambda,\mu,\nu)$ \'etant admissible. 
 Avec les notations pr\'ec\'edentes on obtient 
$$\Delta(P)=(\lambda+\nu+\mu+1)^4x^4y^4(y^2-(\alpha+1)xy+\alpha x^2)^4R(x,y)$$
o\`u $R(x,y)$ d\'esigne le polyn\^ome $r_1x^4+r_2x^3y+r_3x^2y^2
+r_4xy^3+r_5y^4$ avec
\begin{small}
\begin{align*}
r_1=\lambda^2\alpha^2(2\alpha\nu\mu+\alpha^2+2\alpha^2\nu+\alpha^2\nu^2-2\alpha-2\alpha\mu-2\alpha\nu+1+2\mu+\mu^2),
\end{align*}
\begin{align*}
&r_2= -2\lambda\alpha\Big(\lambda\alpha\nu\mu+\lambda\alpha^2\mu\nu-3\alpha^2\mu\nu+\nu\mu^2+\lambda\mu^2-\lambda\alpha\nu+\alpha\nu^2\mu-2\alpha\nu\mu^2+\alpha^2\nu\mu^2-3\alpha\nu\mu-\lambda\alpha\mu-\lambda\alpha^2\mu+2\nu\mu-\alpha^2\mu^2\\
&\hspace{1cm}+\lambda+\nu+2\mu\lambda-\lambda\alpha^2-\lambda\alpha+\lambda\alpha^3+\nu^2\mu\alpha^3+2\nu\mu\alpha^3-\alpha\nu-\alpha^2 \mu+\mu\alpha^3-\lambda\nu\alpha^2-2\mu\nu^2\alpha^2-\nu^2\alpha+2\lambda\nu\alpha^3+\lambda\nu^2\alpha^3\Big),
\end{align*}
\begin{align*}
&r_3= 2\mu\nu^2-4\lambda\alpha^2\mu^2-4\lambda\alpha\nu\mu^2-2\lambda\alpha\nu\mu+2\nu^2\lambda\alpha^2\mu-12\lambda \alpha^2\mu\nu+2\lambda^2\nu\alpha^2\mu+2\lambda\nu\alpha^2\mu^2+4\lambda\nu\mu\alpha^4-4\lambda\nu^2\mu\alpha^3+2 \lambda\nu^2\mu\alpha^4\\
&\hspace{1cm}+4 \lambda\nu\mu+6\nu^2\alpha^2\mu^2-4\lambda^2\alpha^2\nu-4\lambda^2\alpha^2\mu-4\alpha\nu^2\mu+2\alpha^2\nu\mu^2-4\lambda \alpha^2\mu+2\lambda^2\alpha\mu+2\mu\lambda\alpha^4+2\alpha^2\mu\nu-2\lambda\nu\mu\alpha^3+2\lambda\nu\mu^2\\
& \hspace{1cm}+2\lambda\alpha\nu-4\alpha\mu^2\nu^2+2\lambda^2\mu+\lambda^2\mu^2+2\lambda^2\alpha-6\lambda^2\alpha^2+\nu^2+2\lambda\nu+\lambda^2+\mu^2\alpha^4+\lambda^2\alpha^4+2\lambda^2\alpha^3-4 \mu^2 \nu^2\alpha^3+2\lambda^2\nu\alpha^4\\
&\hspace{1cm}+\mu^2 \nu^2-4\lambda\nu\alpha^2-4\nu \mu^2\alpha^3+2\mu\nu^2\alpha^2-4\lambda\nu^2\alpha^2+2\lambda^2\nu\alpha^3+2\mu\lambda\alpha^3+\mu^2\nu^2\alpha^4+2\nu\mu^2\alpha^4+\lambda^2\nu^2\alpha^4,
\end{align*}
\begin{align*}
&r_4=-2\Big(\mu\nu^2-3\lambda\alpha\nu\mu-3\lambda\alpha^2\mu\nu+\alpha^2\mu\nu+2\lambda\nu\mu\alpha^3-\lambda\alpha\nu+2 \lambda\nu\mu-\lambda^2\alpha^2\nu-2\alpha\nu^2\mu+\alpha\nu\mu^2-2\alpha^2\nu\mu^2+\alpha\nu\mu-\lambda\alpha\mu-\lambda\alpha\mu^2-\lambda\alpha^2\mu\\
&\hspace{1cm}+\mu^2\alpha^3+\lambda^2\alpha^3-\lambda\nu\alpha^2+\nu\mu^2\alpha^3+\mu\nu^2\alpha^2-\lambda\nu^2\alpha^2+\lambda^2\nu\alpha^3+2\mu\lambda\alpha^3-\lambda^2\alpha\mu+\lambda^2\mu-\lambda^2\alpha-\lambda^2\alpha^2+\nu^2+2\lambda\nu+\lambda^2\Big),
\end{align*}
\begin{align*}
r_5= \lambda^2+\lambda^2\alpha^2-2\lambda^2\alpha+2\lambda\alpha^2\mu+2\lambda\nu-2\lambda\alpha\nu-2\lambda\alpha\mu+2 \alpha\nu\mu+\nu^2+\alpha^2\mu^2.
\end{align*}
\end{small}

\bigskip

\noindent Consid\'erons l'application $f$ d\'efinie par
\begin{align*}
&\mathbb{C}^4\to \mathbb{P}^4(\mathbb{C}), &&(\alpha,\lambda,\mu,\nu)\mapsto(r_1:
r_2:r_3:r_4:r_5).
\end{align*}
\hspace{-0.3mm}L'ensemble $f^{-1}(\Theta)$ donne les param\`etres 
$(\alpha,\lambda,\mu,\nu)$ pour lesquels le discriminant associ\'e est
un carr\'e; cet ensemble s'identifie donc \`a celui des feuilletages homog\`enes $\mathcal{F}(\alpha;\lambda,\mu,\nu)$ auxquels on peut associer une trivolution. On va d\'ecrire~$f^{-1}(\Theta)$ pr\`es du point sp\'ecial $(-1,1,1,1).$ Pla\c{c}ons-nous dans la carte $\tau_5=1.$ La matrice jacobienne de $f$ au point $(-1,1,1,1)$ est 
$$\mathrm{Jac}(f)_{(-1,1,1,1)}=\left[
\begin{array}{cccc}
-2 & 2/3 & 0 & 0 \\
16/3 & 0 & 2/3 & -2/3 \\
-2 & -14/3 & 16/3 & 16/3\\
-16/3 & 0 & 2/3 & -2/3
\end{array}
\right].$$
Il en r\'esulte que la diff\'erentielle de $f$ au point $(-1,1,1,1)$
est de rang maximum.
 On peut v\'erifier que $f(-1,1,1,1)=(12,0,24,0,12)$ et que si $(\alpha,\lambda,\mu,\nu)$ est admissible et v\'erifie $f(\alpha,\lambda,
\mu,\nu)=(12,0,24,0,12)$ alors $(\alpha,\lambda,\mu,\nu)=(-1,1,1,1).$ 
On remarque que $\Psi^{-1}(12:0:24:0:12)=\{(1:0:1)\}$
et que la diff\'erentielle de $\Psi$ en $(1:0:1)$ est de 
rang $3.$ Dans la carte affine $\tau_1=1,$ l'application $\Psi$ s'\'ecrit
$\widetilde{\Psi}\colon(B,C)\mapsto(2B,B^2+2C,2BC,C^2).$
Le plan tangent \`a $f^{-1}(\Theta)$ au point $(-1,1,1,1)$ est
l'image de 
$$\mathrm{Jac}(f)_{(-1,1,1,1)}^{-1}\mathrm{Jac}(\widetilde{\Psi})_{(0,1)}
=\left[
\begin{array}{cc}
0 & 0 \\
3 & 0\\
\frac{3}{2} & \frac{3}{2}\\
\frac{3}{2} & -\frac{3}{2}
\end{array}
\right].$$
Par suite sur $f^{-1}(\Theta)$ au voisinage de $(-1,1,1,1)$ et \`a l'ordre $1$
on a
\begin{align*}
&\alpha=-1, &&\lambda=1+3u, &&\mu=1+\frac{3}{2}(u+v), && \nu=1+\frac{3}{2}(u-v)
\end{align*}
\hspace{-0.2mm}ou encore $ \alpha=-1,$ et$\nu-1-\lambda+\mu=0.$ En particulier, on en d\'eduit l'\'enonc\'e suivant.

\begin{thm}
{\sl L'espace des param\`etres $(\alpha,\lambda,\mu,\nu)$ admissibles
pour lesquels est associ\'ee \`a $\mathcal{F}(\alpha;\lambda,\mu,\nu)$ une trivolution est localement
au point $(-1,1,1,1)$ une surface lisse.}
\end{thm}

\noindent La trace de cette surface est pr\'ecis\'ee sur la section hyperplane
$\alpha=-1$ par la proposition qui suit.

\begin{pro}\label{soleil}
{\sl On peut associer une trivolution
\`a $\mathcal{F}(-1;\lambda,\mu,\nu)$ si et seulement si le quadruplet 
admissi\-ble~$(-1,\lambda,\mu,\nu)$ satisfait l'une des propri\'et\'es suivantes
\begin{itemize}
\item[(a)] $(4-\nu)\lambda-\nu(2\nu+1)=0$ et $\mu = \nu;$

\item[(b)] $\lambda = 1$ et $\mu+\nu-4\mu\nu+2=0.$
\end{itemize}}
\end{pro}

\begin{rems}\label{cirm}
\begin{itemize}
\item On remarque que $(-1,1,1,1)$ v\'erifie les conditions {\sl (a)} et {\sl (b)}.

\item On passe de la condition {\sl (a)} \`a la condition {\sl (b)} par $(x,y)\mapsto(x+y,y-x),$ la
non homog\'en\'eit\'e apparente r\'esultant du fait qu'on a normalis\'e \`a 
$1$ l'exposant de $x.$
\end{itemize}
\end{rems}

\begin{proof}[{\sl D\'emonstration}]
\noindent Le quadruplet $(-1,\lambda,\mu,\nu)$ est suppos\'e 
admissible. Quitte \`a poser $u=\mu+\nu$ et $w=\mu-\nu$ les $r_i$
s'\'ecrivent
\begin{align*}
&r_1 = \lambda^2 (4u + 4 + w^2), &&
r_2 = 2\lambda w (-u - 2 +\lambda u-w^2 + u^2 +2\lambda), 
\end{align*}
\begin{align*}
&r_3 = u^2 - 2 \lambda u w^2 - 8 \lambda^2 + 2 u^3 - 2 u w^2
- 2 \lambda u^2 - 2 \lambda w^2 - 4 \lambda u + 2 \lambda u^3 - 4 \lambda^2 u
+ \lambda^2 u^2 +(u^2-w^2)^2,
\end{align*}
\begin{align*}
&r_4 =2w(u-\lambda u-w^2+u^2+2\lambda-2\lambda^2), &&
r_5 = 4\lambda^2 + 4\lambda u + w^2.
\end{align*}

\noindent Remarquons que $\lambda+\mu+\nu+1\not=0$ se r\'e\'ecrit 
$\lambda+u+1\not=0.$ Cette pr\'esentation, bien qu'asym\'etrique, 
permet en fait d'\og all\'eger\fg\, les calculs qui vont suivre.

\noindent Pour que la condition $\textbf{\textit{(i)}}$ soit v\'erifi\'ee il faut 
que $r_5$ soit nul autrement dit que $u=-\frac{4\lambda^2+w^2}{4\lambda}.$
Alors~$r_4$ se r\'e\'ecrit
$$\frac{2w}{\lambda^2}\left(\lambda-\frac{w}{2}\right)\left(\lambda+\frac{w}{2}\right)
\left(\lambda-\frac{w^2}{4}\right).$$
Comme le quadruplet $(-1,\lambda,\mu,\nu)$ est admissible (c'est-\`a-dire $\lambda\mu\nu\not=0$ et $\lambda+\mu+\nu+1\not=0$), on v\'erifie que les 
condi\-tions~$r_4=0$ et $4r_1r_3-r_2^2=0$ impliquent que $\lambda=1$ et $\mu=\nu=-1/2.$ Ces valeurs des param\`etres satisfont la condition~{\sl (b)}. Ceci correspond exactement \`a l'exemple \ref{alpha-1b}.

\medskip

\noindent \'Etudions maintenant la condition $\textbf{\textit{(ii)}}$
que l'on traduit par $r_5\not=0,$ $r_4=r_2=r_3^2-4r_1=0.$ Toujours en
invoquant l'admissibilit\'e on v\'erifie que les conditions $r_2=r_4=0$
sont satisfaites pour
\begin{itemize}
\item[1)] $w=0$

ou

\item[2)] $\lambda(u+2)-w^2+u^2-u-2=0$ et $u^2+u-\lambda u-2\lambda^2+2\lambda=0,$
\end{itemize}

\noindent conditions que nous allons examiner cas par cas.

\noindent Si $w$ est nul, {\it i.e.} $\mu=\nu,$ on v\'erifie que 
$r_3^2-4r_1r_5=-u^3(u+\lambda+1)^3((u-8)\lambda+u(u+1)).$
Maintenant l'admissibilit\'e exige la non-nullit\'e de $u$ et $u+\lambda+1;$
de sorte que $r_3^2-4r_1r_5=0$ si et seulement si $(u-8)\lambda+u(u+1)=0.$
Finalement on obtient
\begin{align*}
&\mu=\nu &&\text{et}&&(4-\nu)\lambda-\nu(2\nu+1)=0.
\end{align*}

\noindent On constate ensuite que l'\'eventualit\'e 2) ne produit pas de nouvelles
solutions admissibles.

\medskip

\noindent Pour finir traitons la condition $\textbf{\textit{(iii)}}.$ On constate que $r_4^2r_1-r_5r_2^2=16\lambda^2w^2(\lambda-1)(u+w)(u-w)(\lambda
+u+1)^3.$

\noindent Le quadruplet $(-1,\lambda,\mu,\nu)$ est admissible et $r_4$ est non nul donc $w$ aussi. Il s'en suit que $\lambda$ vaut $1;$ alors 
$4r_3r_4r_5-r_4^3-8r_2r_5^2$ se r\'e\'ecrit
\begin{align*}
32w(u+2)^3(u+w)(u-w)(u^2-w^2-u-2)&&\text{ou
encore}&&
128(\mu-\nu)(\mu+\nu+2)^3\mu\nu(4\mu\nu-\mu-\nu-2).
\end{align*}
Puisque $\mu,$ $\nu$ et $\mu+\nu+\lambda+1$ sont suppos\'es non nuls, $4r_3r_4r_5-r_4^3-8r_2r_5^2=0$ si et seulement si $\mu+\nu-4\mu\nu+2=~0.$
\end{proof}

\noindent D\'ecrivons pr\'ecis\'ement la premi\`ere famille de l'\'enonc\'e \ref{soleil}, la
seconde s'en d\'eduit (Remarque \ref{cirm}). Le feuilletage $\mathcal{F}$ 
est alors d\'efini par le champ
$$\frac{\nu x\Big((2\nu+1)x^2-9y^2\Big)}{4-\nu}\frac{\partial}{\partial x}-y\Big((2\nu+1)x^2-y^2\Big)\frac{\partial}{\partial y}.$$
On v\'erifie que $\Delta(P)$ est un carr\'e
$$\Delta(P)=\frac{1024\nu^2(2\nu+1)^5}{(\nu-4)^6}x^4y^4(x^2-y^2)^4\Big((2\nu+1)x^2+3y^2\Big)^2.$$
Choisissons une racine $\widetilde{\nu}$ de $2\nu+1.$ 
\`A $\mathcal{F}$ est associ\'ee une 
trivolution $\mathcal{T}_\mathcal{F}$ de la forme
$\left(\frac{U_1}{WU_2},\frac{V_1}{WV_2}\right)$ avec 
\begin{small}
\begin{align*}
&U_1=4xy\left(\nu\widetilde{\nu}(2\nu+1)^2x^5
-(12\nu^3+28\nu^2+19\nu+4)x^4y
+2\nu\widetilde{\nu}(\nu-2\nu^2+1)x^3y^2\right.\\
&\hspace{1cm}\left.+2(6\nu^3+13\nu^2+13\nu+4)x^2y^3
-3\nu\widetilde{\nu}(2\nu+1)xy^4+(2\nu^2-7\nu-4)y^5\right),
\end{align*}
\begin{align*}
&V_1=4xy\left(\nu(1+6\nu+12\nu^2+8\nu^4)x^5
+\widetilde{\nu}(4\nu^3-12\nu^2-15\nu-4)x^4y+2\widetilde{\nu}(9\nu^2-3\nu-4-2\nu^3)x^2y^3\right.\\
&\hspace{1cm}\left.-2(45\nu^2+16\nu^3+4\nu^4+8+35\nu)x^3y^2+(20\nu^3+16+69\nu+84\nu^2)xy^4
+3\widetilde{\nu}(4-2\nu^2+7\nu)y^5\right),
\end{align*}
\begin{align*}
&W=\nu^2(2\nu+1)^2x^4-2(2\nu^3+25\nu^2+28\nu+8)x^2y^2+(\nu-4)^2y^4,
&&U_2=(2\nu+1)x^2-y^2,
&&V_2=(2\nu +1)x^2-9y^2.
\end{align*}
\end{small}

\bigskip

\section{Conclusion}

\noindent Les exemples qui pr\'ec\`edent soul\`event plus de probl\`emes qu'ils n'en r\'esolvent et on peut d\'egager un certain nombre de questions. En premier lieu on peut se demander quels sont les feuilletages associ\'es aux involutions de Bertini; ont-ils des propri\'et\'es (dynamiques) sp\'eciales ? Dans un autre ordre d'id\'ees nous avons plusieurs fois mentionn\'e la construction de $3$-tissus associ\'es \`a certains feuilletages produisant des trivolutions. Le fait que ces $3$-tissus soient hexagonaux est-il un fait anecdotique ? Enfin il serait int\'eressant de d\'ecrire compl\`etement la vari\'et\'e des feuilletages de degr\'e $3$ qui produisent des trivolutions. En particulier quel est son degr\'e ? Est-elle irr\'eductible ? Rationnelle ?

\vspace{8mm}

\bibliographystyle{alpha}
\bibliography{biblioinvbirfeuil}

\begin{thebibliography}{CDGBM}

\bibitem[BB00]{BaBe}
L.~Bayle and A.~Beauville.
\newblock Birational involutions of {${\bf P}\sp 2$}.
\newblock {\em Asian J. Math.}, 4(1):11--17, 2000.
\newblock Kodaira's issue.

\bibitem[Bea80]{Be}
A.~Beauville.
\newblock G\'eom\'etrie des tissus [d'apr\`es {S}. {S}. {C}hern et {P}. {A}.
  {G}riffiths].
\newblock In {\em S\'eminaire {B}ourbaki (1978/79)}, volume 770 of {\em Lecture
  Notes in Math.}, pages Exp. No. 531, pp. 103--119. Springer, Berlin, 1980.

\bibitem[Ber77]{Ber}
E.~Bertini.
\newblock Ricerche sulle trasformazioni univoche involutorie nel piano.
\newblock {\em Annali di Mat.}, 8:244--286, 1877.

\bibitem[CdF01]{CdF}
C.~Camacho and L.~H. de~Figueiredo.
\newblock The dynamics of the {J}ouanolou foliation on the complex projective
  2-space.
\newblock {\em Ergodic Theory Dynam. Systems}, 21(3):757--766, 2001.

\bibitem[CDGBM]{CDGBM}
D.~Cerveau, J.~D\'eserti, D.~Garba~Belko, and R.~Meziani.
\newblock G\'eom\'etrie classique des feuilletages quadratiques, {\tt
  arxiv:0902.0877v3}, $2009$.

\bibitem[CLN91]{CLN}
D.~Cerveau and A.~Lins~Neto.
\newblock Holomorphic foliations in {${\bf C}{\rm P}(2)$} having an invariant
  algebraic curve.
\newblock {\em Ann. Inst. Fourier (Grenoble)}, 41(4):883--903, 1991.

\bibitem[CO00]{CO}
A.~Campillo and J.~Olivares.
\newblock Assigned base conditions and geometry of foliations on the projective
  plane.
\newblock In {\em Singularities---{S}apporo 1998}, volume~29 of {\em Adv. Stud.
  Pure Math.}, pages 97--113. Kinokuniya, Tokyo, 2000.

\bibitem[dF04]{dF}
T.~de~Fernex.
\newblock On planar {C}remona maps of prime order.
\newblock {\em Nagoya Math. J.}, 174:1--28, 2004.

\bibitem[DI]{DI}
I.~V. Dolgachev and V.~A. Iskovskikh.
\newblock Finite subgroups of the plane cremona group, {\tt
  arxiv:math/0610595}, 2006.
\newblock {\em Algebra, Arithmetic, and Geometry: Volume I: In Honor of Y.I.
  Manin. \`A para\^itre}.

\bibitem[GMK89]{GMK}
X.~G{\'o}mez-Mont and G.~Kempf.
\newblock Stability of meromorphic vector fields in projective spaces.
\newblock {\em Comment. Math. Helv.}, 64(3):462--473, 1989.

\bibitem[Jou79]{J}
J.~P. Jouanolou.
\newblock {\em \'{E}quations de {P}faff alg\'ebriques}, volume 708 of {\em
  Lecture Notes in Mathematics}.
\newblock Springer, Berlin, 1979.

\bibitem[Per01]{Pe}
J.~V. Pereira.
\newblock Vector fields, invariant varieties and linear systems.
\newblock {\em Ann. Inst. Fourier (Grenoble)}, 51(5):1385--1405, 2001.

\bibitem[Wil20]{Will}
A.~R. Williams.
\newblock On a birational transformation connected with a pencil of cubics.
\newblock {\em University of California Publications in Mathematics},
  20(10):211--222, 1920.

\end{thebibliography}
\nocite{}

\end{document}